\documentclass{amsart}[12 pt]
\usepackage{amsmath, amsfonts, amssymb, amsthm, euscript, amscd, latexsym, mathrsfs, bm, color, bbm, xcolor,mathtools, varwidth, multicol, xfrac, nicefrac}
\usepackage[belowskip=-70pt,aboveskip=0pt]{caption}
 \usepackage{titlesec}
 \usepackage{enumitem}
\usepackage{lipsum}
 \usepackage[pdftex,colorlinks=true,
                       pdfstartview=FitV,
                       linkcolor=blue,
                       citecolor=blue,
                       urlcolor=blue,
           ]{hyperref}

\usepackage{overpic}
\usepackage{epsfig}
\usepackage{graphicx}
\usepackage{graphics}

\def\aa#1{ \begin{align*} #1 \end{align*} }
\def\aaa#1{ \begin{align} #1 \end{align} }
\def\mm#1{ \begin{multline*} #1 \end{multline*} }
\def\mmm#1{ \begin{multline} #1 \end{multline} }

\def\begeq{\begin{equation} \begin{cases}} 
\def\endeq{ \end{cases} \end{equation}}
\def\eq#1{ \begeq #1 \endeq }
\def\bege{\begin{equation*} \begin{cases}} 
\def\ende{ \end{cases} \end{equation*}}
\def\eqq#1{ \bege #1 \ende}
\def\eqn#1{ \begin{equation} \begin{aligned} #1 \end{aligned}\end{equation}}

\newtheorem{thm}{\sc Theorem}
\newtheorem{lem}{\sc Lemma}
\newtheorem{cor}{\sc Corollary}

\newtheorem{pro}{\sc Proposition}
\newtheorem{rem}{\sc Remark}

\newtheorem*{thm*}{\sc Theorem}

\newcommand{\pl}{\partial}
\newcommand{\gt}{\geqslant}
\newcommand{\lt}{\leqslant}

\newcommand{\te}{\theta}
\newcommand{\sub}{\subset}
\newcommand{\dl}{\delta}
\newcommand{\al}{\alpha}
\newcommand{\gm}{\gamma}
 \newcommand{\Gm}{\Gamma}
 
 \newcommand{\la}{\lambda}
 \newcommand{\La}{\Lambda}
 \newcommand{\sg}{\sigma}
\newcommand{\dd}{\diagdown}
\newcommand{\om}{\omega}
\newcommand{\mc}{\mathcal}
\newcommand{\Om}{\Omega}

\newcommand{\mb}{\mathbf}
\newcommand{\C}{{\rm C}}

\newcommand{\td}{\tilde}

\newcommand{\E}{\mathbb E}

\renewcommand{\>}{\rangle}

\newcommand{\x}{\times}
\newcommand{\mto}{\mapsto}
\newcommand{\PP}{\mathbb P}
\newcommand{\W}{\mathbb W}

\newcommand{\cov}{{\rm cov}\,}

\newcommand{\Te}{\Theta}
\newcommand{\rf}{\eqref}
\newcommand{\bi}{\begin{itemize}}
\newcommand{\ei}{\end{itemize}}
\newcommand{\biz}{\begin{itemize}[leftmargin=*]}

\newcommand{\D}{{\mathbb D}}

\DeclareMathOperator{\ind}{\mathbbm{1}}
\newcommand{\lap}{\Delta}
\newcommand{\nab}{\nabla}
\newcommand{\lb}{\label}
\newcommand{\fdot}{\,\cdot\,}

\usepackage{tikz}
\newcommand{\fatdot}{{\raisebox{0.1mm}{\hspace{0.3mm}\tikz\filldraw[black,x=0.4pt,y=0.4pt] (0,0) circle (1);}\hspace{0.2mm}}}

\newcommand{\tet}{\vartheta}

\def\Rnu{{\mathbb R}}

\def\Nnu{{\mathbb N}}

\def\Qnu{{\mathbb Q}}

\def\ffi{\varphi}

\def\com#1{}

\long\def\symbolfootnote[#1]#2{\begingroup%
\def\thefootnote{\fnsymbol{footnote}}\footnote[#1]{#2}\endgroup}
\titleformat{\section}[hang]{\large\bfseries}{\thesection.}{1ex}{}{}
\titleformat{\subsection}[hang]{\normalsize\bfseries}{\thesubsection}{2ex}{}{}
\titleformat{\subsubsection}[hang]{\small\bfseries}{\thesubsubsection}{2ex}{}{}

\include{srctex.sty}

\allowdisplaybreaks

\begin{document}

\title[Gaussian-type density bounds for solutions to multidimensional BSDEs]
{Gaussian-type density bounds for solutions to multidimensional backward SDEs\\and application to gene expression}

\author{Roman Chertovskih}
\address{Research Center for Systems and Technologies, Universidade do Porto, Portugal}
\email{roman@fe.up.pt}

\author{Evelina Shamarova}
\address{Departamento de Matem\'atica, Universidade Federal da Para\'iba, Jo\~ao Pessoa, Brazil}
\email{evelina@mat.ufpb.br}

\maketitle

\vspace{-1cm}

\begin{abstract}
We obtain upper and lower Gaussian-type bounds on the density of each component $Y^i_t$ of the 
solution $Y_t$ to a multidimensional non-Markovian backward SDE.
 Our approach is based on the Nourdin-Viens formula and 
 a stochastic version of Wazewski's theorem on the positivity of the components
 of a solution to an ODE.
 Furthermore, we apply our results to stochastic gene expression; namely, we estimate
 the density of the law of the
amount of protein generated by a gene in a gene regulatory network. 
\end{abstract}

\vspace{4mm}

{\footnotesize

\vspace{1mm}

}

\section{Introduction}
The aim of this work is to find sufficient conditions to ensure that 
each component $Y^i_t$ of the solution $Y_t$ to  the
$m$-dimensional non-Markovian backward SDE  (BSDE)
 \aaa{
 \lb{bsde11:11}
 Y_t = \xi + \int_t^T g(s,Y_s,Z_s) ds + \int_t^T Z_s dB_s \quad 
 }
 admits a density with respect to Lebesgue measure,
and,  furthermore, to obtain upper and lower Gaussian-type bounds on this density.
In addition, we obtain bounds on tail probabilities. 
In \rf{bsde11:11}, $B_t$ is an $n$-dimensional Brownian motion
 and the generator $g: [0,T]\x \Om \x \Rnu^m \x \Rnu^{m\x n}\to \Rnu^m$ 
 can depend on $\om$. 
 As a corollary, we obtain Gaussian-type bounds on the components $Y^i_t$ of
 the solution $Y_t$ to the BSDE
 \aaa{
\lb{bsde}
Y_t = \ffi(X_T) + \int_t^T g(s,X_s,Y_s,Z_s)ds + \int_t^T Z_s dB_s,
}
where $X_t$ is the solution to the $n$-dimensional SDE
\aaa{
 \lb{fsde1}
 X_t = x + \int_0^t f(s,X_s) ds + \int_0^t \sg(s) dB_s.
 }
In \rf{bsde}-\rf{fsde1}, $g: [0,T] \x \Rnu^n \x \Rnu^m \x \Rnu^{m\x n}\to\Rnu^m$, $f:[0,T]\x \Rnu^n \to\Rnu^n$,
  and $\sg: [0,T] \to \Rnu^{n\x n}$ are deterministic functions.

BSDEs
have numerous applications in stochastic control theory, mathematical finance, and biology
(see, for instance, \cite{Ci, Karo, gene,Yong}).
Several recent papers 
 studied the existence of densities and density estimates for the laws of solutions to one-dimensional 
 BSDEs \cite{aboura, antonelli, TM, TM1}, 
 fully coupled one-dimensional
 forward-backward SDEs (FBSDEs) \cite{OS}, and one-dimensional BSDEs  
 driven by Gaussian processes (in particular, by a fractional Brownian motion) \cite{fanwu}.
 To the best of our knowledge, the 
 problem of obtaining density estimates for BSDEs 
in the multidimensional setting is addressed for the first time.
Furthermore, when density estimates are concerned, 
each of the above-cited papers deals only with deterministic generators.

To obtain upper and lower  Gaussian-type bounds, we use the Nourdin-Viens formula  \cite{NV}. 
In the case of  one-dimensional SDEs  and   BSDEs,  driven by a one-dimensional Brownian motion,
bounds  on the density of the solution, say $U_t$,
can be obtained by means of finding  upper and lower positive deterministic
bounds on the expression $\E \Big[\int_0^t D_s  U_t \, \E[D_s U_t |\mc F_s] \, ds | \, U_t\Big]$ \cite{ng, OS},
 where $\mc F_s$ is the (augmented) natural filtration of a one-dimensional Brownian motion. This expression was first
  introduced in \cite{ng}.
By extending the approach of \cite{ng} to the case of
equations driven by a multidimensional
Brownian motion, we show that we are required to obtain upper 
and lower positive  deterministic bounds on the expression 
\aa{
\E \Big[\int_0^t \sum_{k=1}^nD^k_s Y^i_t \, \E[D^k_s Y^i _t|\mc F_s] \; ds\, \big| \, Y^i_t \Big],
}
where $\mc F_s$ is the (augmented) natural filtration of the $n$-dimensional Brownian motion $B_s$.
Unlike the one-dimensional case, where it is usually sufficient to obtain upper and lower positive deterministic bounds on
the Malliavin derivative of the solution  \cite{aboura, ng, NV, OS}, we aim to obtain a positive lower bound on
 each component $D^k_s Y^i_t$  and an upper bound on $|D_s Y^i_t|$.
While upper bounds can be obtained by standard techniques, 
obtaining component-wise lower bounds turns out to be a more delicate task.
Our result in this direction is
inspired by Wazewski's theorem on the positivity of the solution
components for ODEs \cite{Wazewski}. More specifically, we consider
the system of linear BSDEs
\aaa{
\lb{linear-bsde}
U^i_t = \xi^i + \int_t^T (K^i_s  + \sum_{j=1}^m F^{ij}_s U^j_s +
\sum_{j=1}^n G^{ij}_s V^{ij}_s) ds + \sum_{j=1}^n \int_t^T V^{ij}_s dB^j_s,
}
$i=1,\ldots, m$. Assuming the non-negativity of $\xi^i$, $K^i_t$ and  
$F^{ij}_t$ for $i\ne j$, along with additional standard assumptions on the coefficients,
we obtain the estimate
\aaa{
\lb{u-est-1111}
U^i_t \gt \E_{\Qnu^i}\Big[\xi^i e^{\int_t^T F^{ii}_s ds}
+ \int_t^T   e^{\int_t^s F^{ii}_r dr} K^i_s ds \,\big|\,\mc F_t\Big],
}
where $\Qnu^i$ is a probability measure absolutely continuous with respect
to the original measure $\PP$. Remark 
that the existing explicit formulas for solutions to multidimensional linear BSDEs
\cite{DelbaenTang, HarterRichou} 
do not allow to obtain component-wise lower bounds.

Note that the $i$-th equation in system \rf{linear-bsde} depends only on the $i$-th line
of the matrix $\{V^{ij}_s\}$. This restricts the class of generators $g(t,y,z)$ to those
whose $i$-th component depends only on the $i$-th line of the matrix $z$. 
However, we provide the following example of a two-dimensional 
BSDE driven by a one-dimensional Brownian motion:
\eq{
\lb{exl-1111}
Y^1_t = e^{-T}B_T - \int_t^T Z^2_s ds + \int_t^T Z^1_s dB_s,\\
Y^2_t = 2B_T - \cos B_T  + \int_t^T Z^2_s dB_s.
}
For this BSDE we show that (i) for a large interval of values of $t\in (0,T)$, 
the Malliavin derivative $D_r Y^1_t$ does not take only positive or only negative
values; 
(ii) at some points $t\in (0,T)$,
the density of $Y^1_t$  with respect to Lebesgue measure does 
not possess Gaussian-type bounds. Talking about property (i), 
we would like to emphasize that the strict positivity of the Malliavin derivative
 of a random variable (in particular of an SDE or a BSDE solution)
is a common requirement to obtaining Gaussian-type density bounds
by means of the Nourdin-Viens formula \cite{NV}; see, for instance, \cite{aboura, ng, TM, TM1, NV, OS}.
Example \rf{exl-1111} shows that the above-described class of generators, i.e.,
when $g^i(t,y,z) = g^i(t,y,z^i)$, where $z^i$ is the $i$-th line of the matrix $z$,
 is the most general one for which the results 
on Gaussian-type bounds for one-dimensional BSDEs (e.g. \cite{aboura, TM, OS}) 
can be extended to the multidimensional case.

Obtaining Gaussian-type bounds on the densities for the BSDE \rf{bsde} requires, in particular, lower bounds on the components  $D^k_s X^i_t$   of the Malliavin
drivatives of the components  of the solution to 
the forward SDE \rf{fsde1}.
Since the diffusion coefficient $\sg(t)$ is assumed to be independent of the solution $X_t$, 
 $D^k_s X^i_t$'s 
solve a system of ODEs. To obtain non-negative lower bounds on these components,
we employ Wazewski's theorem \cite{Wazewski}. 
Remark that our restriction on the diffusion coefficient is only due to the multidimensional setting.
 If the forward SDE is one-dimensional, the diffusion coefficient may also depend on $X_t$,
and one can apply the Lamperti transform, as it is described, e.g., in \cite{OS}.

Furthermore, we apply our results to obtain Gaussian-type bounds on the density of  the law of the
protein level of a gene which is a part of a gene regulatory network.
To model stochastic gene expression,
we employ the backward SDE approach developed in \cite{gene}.
Our results apply to a network consisting of more than one gene, 
and therefore, stochastic gene expression
is modeled by a multidimensional BSDE. 
In addition, we obtain upper
and lower bounds on tail probabilities which allows to compute prediction intervals for 
expression of individual genes. 
More specifically, in gene expression models, one usually deals
with an ensemble of cells, where each cell contains a gene that expresses a certain protein.
From this point of view, the protein amount of each gene at time $t$ becomes a random variable.
 Thus, one wants 
to know the interval in which protein amounts of identical genes will fall if we measure them 
from different cells.
Also, using tail probabilities, one can prove that 
protein amounts
remain positive by showing that the probabilities that they are non-positive are negligibly small. 
Remark that it is not possible to prove that the amounts of proteins are always positive since
this would mean that their densities do not possess Gaussian-type bounds.

For one self-regulating gene, whose expression was also modeled
by the BSDE method \cite{gene},
the problem of existence of a density and Gaussian-type bounds on this density
 was studied in \cite{TM1}.
In addition, in \cite{TM1}, a numerical simulation was performed so one could observe that the density estimates agree
with the data produced by the BSDE method. However, the approach used \textit{ibid.}
is essentially one-dimensional, and, therefore, can only be applied to a self-regulating gene.
Our approach allows to deal with gene regulatory networks. It can also be applied
to a self-regulating gene; however, it is not the goal of the present work.

Finally, we performed a numerical simulation
for a particular type of a gene regulatory network with the purpose 
 to demonstrate that our density estimates agree with the benchmark data 
generated by Gillespie's algorithm \cite{Gil}.
 The latter fact, as a byproduct, can be regarded as another evidence 
of the validity of the BSDE approach \cite{gene} as a tool to model stochastic gene expression.

The organization of our paper is as follows. In Section 2, we give
some necessary preliminary results. In Section 3, we obtain a ``stochastic version'' of Wazewski's theorem;
namely,
we provide sufficient conditions for the non-negativity of the solution components 
for linear BSDEs of type \rf{linear-bsde} and obtain the lower bound \rf{u-est-1111}.
In Section 4, we obtain Gaussian-type bounds on the density of each component $Y^i_t$ of the solution  $Y_t$
to both multidimensional BSDEs, \rf{bsde11:11} and \rf{bsde}. In the same section we demonstrate
that if the last argument of $g^i(t,y,z)$ contains not only the $i$-th line of the matrix $z$, then
the density of $Y^i_t$ may not have Gaussian-type bounds at some points $t$.
 Section 5 is dedicated to applications of the theoretical results  to gene expression. 
 Namely,
 we apply the results of Section 4 to obtain Gaussian-type bounds
on the density of the law of the amount of protein generated by a gene in a gene regulatory network.
Furthermore, using tail probabilities, we compute, 
theoretically and numerically, prediction intervals for expressions
of individual genes. Finally,
in the same section, we describe results of a numerical simulation which show
that the data obtained by Gillespie's algorithm fit between the theoretically determined density bound
curves.

\section{Preliminaries}

\subsection{Malliavin derivative}

Here we describe some elements of the Malliavin calculus  that we  need in the paper.  
We refer to \cite{nualart} for a more complete exposition. 

Let $\mc H$ be a real separable Hilbert space and $\W (h)$, $h \in \mc H$,  be an isonormal Gaussian process on a probability space 
$(\Om, \mc F, \PP)$, which is a centered Gaussian family of random variables with the property
$\E \big( \W(h_1)\W(h_2) \big) = \<h_1, h_2\>_{\mc H}$.

We denote by $D$  the Malliavin  derivative operator. It is known that (see, e.g., \cite{nualart}) if $F$ is a smooth random variable of the form $F=g(\W(h _1), \ldots , \W(h_k))$, where $g$ is a smooth compactly supported function and $h_i \in \mc H$, $i=1,...,k$, then
\aa{
DF=\sum_{i=1}^{k} \pl_{x_i} g(\W(h_1), \ldots , \W(h_k))h_i.
}
It can be shown that the operator $D$ is closable from the space $\mc S$ of smooth random variables of the above form
 to $ L^2(\Om, \mc H)$ and can be extended to the space $\mathbb{D}^{1,p}$ which is the 
closure of $\mc S$ with respect to the norm
\begin{equation*}
\Vert F\Vert _{1,p}^{p} = \mathbb{E} |F|^{p} + \mathbb{E} \Vert DF\Vert _{\mathcal{H}}^{p}. 
\end{equation*}
In our paper, $\mc H= L_2([0,\infty),\Rnu^n)$ and $\W(h) = \sum_{i=1}^n\int_0^\infty h^i(t) dB^i_t$, where $B^i_t$ are independent real-valued 
standard Brownian motions.

Furthermore, $\dl$ denotes the Skorokhod integral, and  $L=-D\dl$ denotes the Ornstein-Uhlenbeck operator.

\subsection{Formula for the density}
Let $(\Om, \mc F, \PP)$ be a probability space.
Following \cite{NV}, we define 
\aa{
g_F(x) = \E\big[ \big(D_rF,-D_rL^{-1} F\big)_{\mc H}| \, F=x \big], \quad x\in\Rnu,
}
and
recall the criterium for the existence of the density of a real-valued  zero-mean random variable $F$ 
and the explicit formula for this density
\cite{NV}.
\begin{pro}
\lb{pro11}
Let $F\in  \D^{1,2}$, $\E F = 0$. Then, the law of $F$ has a density $\rho_F$ with respect to Lebesgue measure
if and only if 
$g_F(F)>0$. In this case, ${\rm supp\,} (\rho_F)$ is a closed interval in $\Rnu$ containing $0$
and for almost all $x\in {\rm supp\,} (\rho_F)$
\aa{
\rho_F(x) = \frac{\E|F|}{2g_F(x)} \exp \Big( -\int_0^x \frac{y}{g_F(y)} dy\Big).
}
\end{pro}
Further, in \cite{ng}, the authors showed that if $\mc H = L_2([0,\infty),\Rnu)$ 
and the associated isonormal Gaussian process is $\W(h) = \int_0^\infty h(t) dB_t$, where $B_t$ is a
one-dimensional Brownian motion, then, a.s., 
\aaa{
\lb{gf}
g_F(F) = \ffi_F(F), 
}
where 
$\ffi_F(x) =  \E \Big[\int_0^\infty D_r F \, \E[D_r F |\mc F_r] \, dr | \, F=x \Big]$,  $x\in\Rnu$,
and $\mc F_t$ is the 
filtration generated by the Brownian motion $B_t$, $t\in \Rnu_+$,  and augmented with $\PP$-null sets.

\section{Positivity of the components of solutions to linear BSDEs}
The result of this section is inspired by Wazewski's theorem \cite{Wazewski},
and can be regarded as a generalization of the latter to linear backward SDEs.

We start by announcing the Wazewski theorem.
Consider the Cauchy problem for a system of linear ordinary differential equations
\aaa{
\lb{ode}
u'(t) = A(t)u(t), \qquad u(0) = u_0, \quad t\in [0,\infty),
}
where $A(t)$ is an $n\x n$ matrix. 
Below, $u^i(t)$ and $u^i_0$ denote the $i$-th coordinates of the vectors $u(t)$ and $u_0$, respectively.
The theorem reads:
\begin{pro}
\lb{pro14}
Let $a_{ij}: [0,\infty) \to \Rnu$ be continuous functions, $i,j=1,\ldots,n$, and 
$u(t)$ be the solution to problem \rf{ode}. Further let $u^i_0 \gt 0$ for all $i=1,\ldots,n$. 
Then, the following two conditions are equivalent:
\bi
\item[1)] for all  $i,j=1,\ldots,n$ such that $i\ne j$, $a_{ij}(t) \gt 0$ on $(0,\infty)$; 
\item[2)] $u^i(t)\gt 0$ for all $t\in (0,\infty)$ and $i=1,\ldots,n$.
\ei
\end{pro}
  \begin{cor} 
  \lb{pro24}
  Let the assumptions of Proposition \ref{pro14} be fulfilled and let condition 1) of the same proposition
  be in force.
  Then, for all $i = 0,\ldots, n$ and for all $t\in (0,\infty)$
  \aaa{
  \lb{est-1}
  u^i(t) \gt u^i_0 \exp\big\{\int_0^t a_{ii}(s) ds\big\}.
  }
  \end{cor}
  \begin{proof}
  For the $i$-th coordinate of the solution $u$ we have
  \aa{
  \frac{d u^i(t)}{dt} =  \sum_{j=1}^n a_{ij}(s) u^j(t) \gt   a_{ii}(t) u^i(t). 
  }
  Indeed, Proposition \ref{pro14} implies that $\sum_{j\ne i} a_{ij}(t) u^j(t)  \gt 0$. Inequality \rf{est-1} follows
  now from (the differential from of) Gronwall's inequality.
  \end{proof}
  Note that the backward form of inequality \rf{est-1} is 
  \aa{
  v^i(t) \gt v^i(T) \exp\big\{\int_t^T a_{ii}(T-s) ds\big\}, \quad v^i(t) = u^i(T-t).
  }
  We now obtain an analog of Corollary \ref{pro24}, and in particular of the above inequality,
for systems of linear BSDEs of the form \rf{linear-bsde}. Let
$(\Om, \mc F, \mc F_t, \PP)$ be a filtered probability space, where $\mc F_t$
 is the natural filtration of  an $n$-dimensional Brownian motion 
 $B_t$ augmented with $\PP$-null sets. 
Let $U_t$ denote the vector $(U^1_t,\ldots, U^m_t)$ and $V_t$ denote the $\Rnu^{m\x n}$-matrix
$\{V^{ij}_t\}_{i=1, j=1}^{m\quad n}$.
\begin{thm}
\lb{t1111}
For the coefficients of  equation \rf{linear-bsde} we assume
\bi
\item[(i)] $F^{ij}_t$, $G^{ij}_t$, $K^i_t$,  $i=1,\ldots, m$, $j=1,\ldots, n$
 are $\mc F_t$-predictable; $\xi^i\in L_2(\Om)$ is $\mc F_T$-measurable;
 $K^i_t\in L_2(\Om\x [0,T])$.
  \item[(ii)] $\xi^i, K^i_t \gt 0$ and $F^{ij}_t\gt 0$ a.s. for $i\ne j$.
\item[(iii)] $F^{ij}_t$ and $G^{ij}_t$ are bounded in $(t,\om)$ a.s.
for all $i=1,\ldots, m$, $j=1,\ldots, n$.
\ei
Then, for each $i=1,\ldots, m$, on $(\Om,\mc F, \mc F_t)$
there exists a probability measure $\Qnu^i$, absolutely continuous
with respect to $\PP$, such that for each component $U^i_t$  of the 
solution to the BSDE \rf{linear-bsde}, estimate \rf{u-est-1111} holds $\PP$-a.s.
\end{thm}
\begin{rem}
\rm
Note that according to \cite{Karo} (Theorem 2.1), under assumptions (i) and (iii), there exists 
an $\mc F_t$-adapted pair  $(U_t,V_t)$ which solves  \rf{linear-bsde} and such that
$U_t$ is continuous and $V_t$ is $\mc F_t$-predictable. This pair is unique in
the $L_2(\Om\x [0,T])$-norm.
\end{rem}
\begin{proof}[Proof of Theorem \ref{t1111}]
Consider the system of BSDEs for $i=1,\ldots, m$
\aaa{
\lb{bsde-1111}
U^i_t = \xi^i + \int_t^T \Big(K^i_s  + \sum_{j=1, j\ne i}^m F^{ij}_s |U^j_s| + F^{ii}_s U^i_s +
(G^{i}_s, V^i_s)\Big) ds +  \int_t^T (V^i_s, dB_s), 
}
where $G^i_s = (G^{i1}_s, \ldots, G^{in}_s)$, $V^i_s = (V^{i1}_s, \ldots, V^{in}_s)$.

First, let is observe that system \rf{bsde-1111}  also possesses an $\mc F_t$-adapted solution $(U_t,V_t)$
such that $U_t$ is continuous and $V_t$ is $\mc F_t$-predictable. 
This, in particular, follows from the inequality $||a|-|b||\lt |a-b|$, $a,b\in\Rnu$.
Moreover,
this pair is unique in the  $L_2(\Om\x [0,T])$-norm (see \cite{Karo}, Theorem 2.1).

Next, by the boundedness of $G^{ij}$ and the
multidimensional version of Girsanov's theorem (see, e.g., \cite{jeanblanc}, Section 1.7), 
for each fixed $i\in\{1,\ldots, m\}$, $\td B_{t,i} = B_t - \int_0^t G^i_s ds$
is a Brownian motion under the probability measure $\Qnu^i$ defined as
\aa{
\Qnu^i|_{\mc F_t} = L^i_t\, \PP|_{\mc F_t}, \qquad
L_t^i = \exp\Big\{\int_0^t (G^{i}_s, dB_s) -  \frac12 \int_0^t |G^{i}_s|^2 ds\Big\}.
}
Under the probability measure $\Qnu^i$, 
the $i$-th  BSDE in \rf{bsde-1111} transforms to 
\aa{
U^i_t = \xi^i + \int_t^T (K^i_s  + \sum_{j=1, j\ne i}^m F^{ij}_s |U^j_s| + F^{ii}_s U^i_s) ds 
+  \int_t^T (V^i_s, d\td B_{s,i}),
}
Applying It\^o's formula to $e^{\int_0^t F^{ii}_s ds} U^i_t$, noticing  that the term 
$e^{\int_0^s F^{ii}_r dr} F^{ii}_s U^i_s$  cancels with the equal one,
 and taking the conditional expectation
$\E_{\Qnu^i}[\fdot|\mc F_t]$ with respect to the measure $\Qnu^i$, we obtain
\mmm{
\lb{u1111}
U^i_t e^{\int_0^t  F^{ii}_s ds} 
= \E_{\Qnu^i}\Big[ \xi^i e^{\int_0^T  F^{ii}_s ds} +
 \int_t^T   e^{\int_0^s F^{ii}_r dr} (K^i_s  + \sum_{j=1, j\ne i}^m F^{ij}_s |U^j_s| ) ds \,\big|\,\mc F_t \Big]\\
\gt \E_{\Qnu^i}\big [ \xi^i e^{\int_0^T  F^{ii}_s ds} 
+ \int_t^T   e^{\int_0^s F^{ii}_r dr} K^i_s \, ds \, \big |\, \mc F_t\big] \quad \text{a.s.}
}
This immediately implies \rf{u-est-1111} for the solution of the modified BSDE
\rf{bsde-1111}. Thus, $U^i_t\gt 0$ for all $i$ a.s., and hence,
$(U^1_t, \ldots, U^m_t)$ is also a solution to \rf{linear-bsde}.  By uniqueness, we conclude that 
\rf{u-est-1111}  holds for all components of the solution to the BSDE
\rf{linear-bsde}.
\end{proof}
\begin{cor}
\lb{cor-qi1111}
If, under the assumptions of Theorem \ref{t1111}, each $G^{ij}_t$ identically equals zero
(i.e., the generator does not depend on $V_t$), then, in \rf{u-est-1111}, $\Qnu^i = \PP$ for all $i$.
\end{cor}
 
\section{Main result}
Let, as before, $(\Om, \mc F, \mc F_t, \PP)$ be a filtered probability space, where $\mc F_t$
 is the natural filtration of an $n$-dimensional Brownian motion $B_t$ augmented with $\PP$-null sets. 

 Consider the BSDE \rf{bsde11:11} which we rewrite as a system:
 \aaa{
 \lb{bsde1111}
 Y^i_t = \xi^i + \int_t^T g^i(s,Y_s,Z_s) ds + \int_t^T (Z^i_s, dB_s), \quad i=1,\ldots, m.
 }
 Here, $\xi=(\xi^1,\ldots, \xi^m)$ is an $\mc F_T$-measurable random variable, 
 $g^i$ are components of $g$, and 
 $Z^i_t = (Z^{i1}_t, \ldots, Z^{in}_t)$ is the $i$-th line of $Z_t$. 
 We aim to find sufficient conditions when $Y^i_t$ possesses a density with respect to Lebesgue
 measure and find Gaussian-type bounds on this density.
 
 \subsection{Gaussian-type density bounds in the situation when the Malliavin derivative is
 multidimensional}

We will need an ``$n$-dimensional'' version of \rf{gf}. Namely, if $\mc H = L_2([0,\infty),\Rnu^n)$ and
$\W(h) = \sum_{i=1}^n \int_0^{+\infty} h_i(t) dB^i_t$, where $B^i_t$ are independent real-valued  standard Brownian motions, 
then $g_F(F) = \ffi_F(F)$, where 
\aaa{
\lb{ffiF}
\ffi_F(x) =  \E \Big[\int_0^\infty \sum_{i=1}^nD^i_t F \, \E[D^i_t F |\mc F_t] \; dt\, \big| \, F=x \Big], \quad x\in\Rnu.
}
Proposition \ref{pro2} below is an extension of Proposition 2.3 in \cite{ng} to the case of 
$\mc H =  L_2([0,\infty),\Rnu^n)$.
\begin{pro}
\lb{pro2}
Let $F\in \D^{1,2}$ be a random variable such that $\E[F] = 0$ and
$\E\int_0^\infty \|D_s F\|^2 ds <\infty$.
Then, $g_F(F) = \ffi_F(F)$ a.s., where $\ffi_F(x)$ is defined by \rf{ffiF}.
\end{pro}
\begin{proof}
First, we prove that for the covariance $\cov(F,G)$ of two random variables $F,G\in \D^{1,2}$, it holds that
\aaa{
\lb{cov}
\cov(F,G) = \E\Big[ \int_0^\infty\sum_{i=1}^nD^i_s G \, \E[D^i_s F |\mc F_s] \; ds \Big].
}
Remark that in the case $n=1$, \rf{cov} was obtained in \cite{privault} (Proposition 3.4.1).
We show \rf{cov} for $n>1$. By the Clark-Ocone formula  (see, e.g., \cite{privault}, p. 171), 
\aa{
F = \E F + \sum_{i=1}^n \int_0^{+\infty} \E[D^i_sF \,| \,\mc F_s] dB^i_s.
}
Therefore,
\mmm{
\lb{f-cov}
\cov(F,G) = \E\big[ (F-\E F)(G-\E G)\big] \\  = \E\Big[\sum_{i=1}^n \int_0^\infty \E [ D^i_sF \,| \,\mc F_s] dB^i_s
\sum_{i=1}^n \int_0^\infty \E[D^i_sG \,| \,\mc F_s] dB^i_s\Big] \\
 =\E\int_0^\infty\sum_{i=1}^n \E\big[ \E[ D^i_sF \,| \,\mc F_s] D^i_s G| \mc F_s \big] \; ds
=  \E\Big[ \int_0^\infty\sum_{i=1}^nD^i_s G \, \E[D^i_s F |\mc F_s] \; ds \Big].
}
Now by formula (3.15) in \cite{NV}, for any $\C^1_b$-function $f:\Rnu\to\Rnu$,
\aa{
\cov(F,f(F)) = \E[F\, f(F)] = \E[f'(F) g_F(F)].
}
On the other hand, by \rf{f-cov},
\aa{
\cov(F,f(F)) = \E\Big[ \int_0^\infty\sum_{i=1}^n f'(F) D^i_s F \, \E[D^i_s F |\mc F_s] \; ds \Big] = \E[f'(F) \ffi_F(F)].
}
Take the function $f(x) = \int_0^x \ind_B(y) dy$, where $B\sub\Rnu$ is a Borel set. By approximating $f$ by mollifiers, and then
passing to the limit, we obtain the identity
$\E[\ind_B(F) g_F(F)] = \E[\ind_B(F) \ffi_F(F)]$, 
or, which is the same, $\int_B g_F(x) \mu_F(dx) = \int_B \ffi_F(x) \mu_F(dx)$, where $\mu_F = \PP \circ F^{-1}$. 
Therefore, $g_F(F) = \ffi_F(F)$ a.s.
\end{proof}
The following corollary will be used  for establishing the bounds on the density of $F$
if we know lower bounds on the components $D^k F$ and an upper bound on $|DF|$.
\begin{cor}
\lb{cor-main}
Let $F\in\D^{1,2}$ be $\mc F_t$-measurable.
Assume there exist functions $[0,t]\to (\Rnu_+)^n$, $r\mto m_{r,t}=(m^1_{r,t},m^2_{r,t},\ldots, m^n_{r,t})$  
and $[0,t]\to (0,\infty)$, $r\mto M_{r,t}$ such that for each $k=1,\ldots, n$, 
$D^k_{\!\fatdot} F \gt m^k_{\fatdot,t}$  a.s.,
and, moreover,  $|D_{\!\fatdot} \,F|\lt M_{\fatdot,t}$  a.s.
Further assume that $\la(t) = \int_0^t |m_{r,t}|^2 dr>0$ and $\La(t) =  \int_0^t M^2_{r,t} \, dr<\infty$.
Then, $F$ admits a density $\rho_F$ w.r.t. Lebesgue
measure, and for almost all $x\in\Rnu$, 
\aa{     
\frac{\E |F-\E[F]|}{2\La(t)} \exp\Big(-\frac{\big(x-\E[F]\big)^2}{2\la(t)}\Big) \lt 
\rho_F(x) \lt \frac{\E |F-\E[F]|}{2\la(t)} \exp\Big(-\frac{\big(x-\E[F]\big)^2}{2\La(t)}\Big).
}
Furthermore, for all $x>0$, the tail probabilities satisfy
\aa{     
\PP(F \gt \E[F]+x) \lt \exp\Big(-\frac{x^2}{2\La(t)}\Big) \quad \text{and}  \quad
\PP(F \lt \E[F]-x) \lt \exp\Big(-\frac{x^2}{2\La(t)}\Big).
}
\end{cor}
\begin{proof}
By \rf{ffiF},  it holds that $\la(t) \lt \ffi_F(F) \lt \La(t)$ a.s.
By Proposition \ref{pro2} and Corollary 3.5 in \cite{NV}, the law of $F$ has a density $\rho_F$ w.r.t. Lebesgue measure
and the above estimate for the density $\rho_F$ holds true. 
The estimates on the tail probabilities follow from Theorem 4.1 in \cite{NV}.
\end{proof}

\subsection{Theorems on Gaussian-type bounds for multidimensional BSDEs}
\lb{sec1111}
First we would like to ensure the existence and the Malliavin differentiability of the solution to \rf{bsde1111}.  
To this end, we introduce assumptions (A1)--(A6) (see \cite{Karo}, Theorem 2.1 and Proposition 5.3):
 \bi
 \item[\bf (A1)] 
 $g: [0,T] \x \Om  \x\Rnu^m \x \Rnu^{m\x n} \to\Rnu^m$ possesses uniformly bounded
 continuous partial derivatives $\pl_y g(t,y,z)$ and $\pl_z g(t,y,z)$.
  \item[\bf (A2)]  For each $(y,z)$, $[0,T]\x \Om\to\Rnu^m$,  $(\om,t)\mto g(t,y,z)$ is an 
  $\mc F_t$-predictable process and $g(\fdot,0,0)\in L_2([0,T]\x \Om)$.
  \ei
 Remark that under (A1) and (A2), there exists an $\mc F_t$-adapted pair $(Y_t,Z_t)$ solving 
 \rf{bsde1111} and such that $Y_t$ is continuous and $Z_t$ is $\mc F_t$-predictable; 
the pair is unique with respect to the norm
 $\E \sup_{[0,T]}|Y_t|^2 + \E\int_0^T|Z_t|^2 dt$
  (see \cite{Karo}, Theorem 2.1).
For the Malliavin differentiability of $(Y_t,Z_t)$, we assume (A3)--(A6) (\cite{Karo}, Proposition 5.3).
  \bi
   \item[\bf (A3)] For all $(t,y,z)$,
   $g(t,y,z)\in \mathbb D^{1,2}$ and
the map  $\Om \x [0,T]\to (L_2[0,T])^{m\x n}$, $(\om,t) \mto D g(t,y,z)$
 admits an $\mc F_t$-predictable version.
  \item[\bf (A4)]  $\xi \in L_4(\Om)\cap  \mathbb D^{1,2}$.
  \item[\bf (A5)] $\int_0^T \! |g(s,0,0)|^2 ds \in L_2(\Om)$;
  $D g(\fdot,0,0), D g(\fdot,Y_{\fdot},Z_{\fdot}) \in L_2([0,T]^2 \x \Om)$; 
   \item[\bf (A6)] For all $(y_1, z_1, y_2, z_2)$ and $t\in [0,T]$,
   \aa{
  |D_r g(t,\om,y_1,z_1)  - D_rg(t,\om,y_2,z_2)|
  \lt K(r,t,\om) (|y_1-y_2| + |z_1 - z_2|)
  }
  where for a.e. $r$,  
  $K(r,t,\fdot)$ is an $\Rnu_+$-valued $\mc F_t$-adapted process such that
  $K(\fdot)\in L_4([0,T]^2 \x \Om)$.
  \ei
  According to \cite{Karo}, Proposition 5.3, under (A1)--(A6), $Y_t,Z_t\in \mathbb D^{1,2}$ and 
  there is a version of $\{(D_r^k Y_t, D^k_r Z_t), 0\lt r, t \lt T\}$ satisfying the BSDE
  \mmm{
\lb{mal-1111}
D^k_r Y_t=  D^k_r \xi +  \int_t^T \big[ D^k_r g(s, Y_s, Z_s)   +\pl_y g(s,Y_s,  Z_s)D^k_r Y_s  \\
+\pl_z g(s,Y_s,  Z_s)D^k_r Z_s\big] ds 
+ \int_t^T  D^k_r Z_s dB_s \quad  \text{if} \;\; t\gt r,
}
and $D_r^k Y_t = 0$, $D^k_r Z_t = 0$, if $t<r$.
Here, $D^k_r g(s, Y_s, Z_s) =  D^k_r g(s, y, z)\big|_{y=Y_s, z=Z_s}$.
 Finally, assumptions (A7)--(A10) are necessary to 
 guarantee the existence of 
 a lower bound on each component $D^k_r Y^i_t$ and 
 an upper bound on $|D_r Y_t^i|$.
\bi
\item[\bf (A7)] Each component $g^i(t,y,z)$ of $g$ depends only on 
$z^i=(z^{i1}, \ldots, z^{in})$ in the last argument. 
\item[\bf (A8)] 
There exist non-negative functions $\beta_{ik}(r)$, $i=1,\ldots, m$, $k=1,\ldots, n$, 
with the property $\int_0^t \big(\sum_{k=1}^n \beta_{ik}^2(r)\big) dr > 0$
and a positive square-integrable function $\Phi(r)$ such that 
a.s. $D_r^k \xi^i \gt  \beta_{ik}(r)$ and $|D_r \xi|\lt \Phi(r)$, $r\in [0,T]$.
\item[\bf (A9)]  For all $t,y,z$, a.s.,  $\pl_{y_j}g^i(t,y,z) \gt 0$ for $i\ne j$.
\item[\bf (A10)] 
There exists a version of $r\mto D_r g(t,y,z)$ such that
for all $k,i$ and $r,t\in [0,T]$, a.s., $D^k_r g^i(t, y, z)\gt 0$ and $\E \big[\int_t^T |D_r g(s,Y_s,Z_s)|^2 ds\, \big | \, \mc F_t\big]\lt \Phi^2(r)$, where $\Phi(r)$ is as in (A8).
\ei
Theorem \ref{t-main1111} and Corollary \ref{cor1111} below are our main results.
\begin{thm} 
\lb{t-main1111}
Let assumptions (A1)-(A10) be fulfilled. Further let $Y_t$ be the first component of the
unique $\mc F_t$-adapted solution to the BSDE \rf{bsde1111} (whose existence is known under 
(A1) and (A2)).
Then, 
there exists a density $\rho_{Y^i_t}$ of $Y^i_t$ w.r.t. Lebesgue measure. Moreover, 
there are positive functions $\la_i(t)$, $i=1,\ldots, m$, and $\La(t)$, $t\in [0,T]$, that can be computed explicitly, such that
for almost all $x\in\Rnu$,
\mmm{
\lb{est1111}     
\frac{\E |Y^i_t-\E Y^i_t|}{2\La(t)} \exp\Big(\!-\frac{\big(x-\E Y^i_t\big)^2}{2\la_i(t)}\Big) \lt 
p_{Y^i_t}(x)  \\ \lt \frac{\E |Y^i_t-\E Y^i_t|}{2\la_i(t)} \exp\Big(\!-\frac{\big(x-\E Y^i_t\big)^2}{2\La(t)}\Big).
}
Furthermore, for all $x>0$, the tail probabilities satisfy
\eqn{ 
\lb{est2222}    
\PP(Y^i_t \gt \E Y^i_t +x) \lt \exp\Big(\!-\frac{x^2}{2\La(t)}\Big),\\
\PP(Y^i_t \lt \E Y^i_t-x) \lt \exp\Big(\! -\frac{x^2}{2\La(t)}\Big).
}
\end{thm}
To formulate sufficient conditions for the existence of densities  and
Gaussian-type density bounds 
for the components 
of the solution $Y_t$ to the FBSDE \rf{bsde}-\rf{fsde1}, below we present another set of assumptions,
implying (A1)--(A10) for the particular case of $g$ which
becomes random through the dependence on the forward component $X_t$
of the solution.

\bi
\item[\bf (H1)] $f: [0,T]\x \Rnu^n \to \Rnu^n$ possesses a bounded partial derivative
$\pl_x f(t,x)$.
\item[\bf (H2)] 
 $g: [0,T] \x \Rnu^n \x\Rnu^m \x \Rnu^{m\x n} \to\Rnu^m$, $g(t,x,y,z)$, possesses 
 bounded continuous partial derivatives $\pl_y g$ and $\pl_z g$, and a bounded derivative $\pl_x g$. 
 \item[\bf (H3)] $g(\fdot, 0,0,0)\in L_2([0,T])$ and for each $(y,z)$, $(t,x)\mto g(t,x,y,z)$  
 is a Borel function.
 \ei
 Note that (H1)--(H3) imply (A1) and (A2) for the BSDE \rf{bsde}. 
 Indeed, under (H1), there exists a unique $\mc F_t$-adapted continuous solution $X_t$
 to  \rf{fsde1}.
Furthermore, there exists an $\mc F_t$-adapted pair $(Y_t,Z_t)$ solving 
 \rf{bsde} and such that $Y_t$ is continuous and $Z_t$ is $\mc F_t$-predictable.
 The pair is unique with respect to the squared norm
 $\E \sup_{[0,T]}|Y_t|^2 + \E\int_0^T|Z_t|^2 dt$
  (see \cite{Karo}, Theorem 2.1).
 We further assume 
  \bi
  \item[\bf (H4)] $\sg \in L_4([0,T])$.
  \item[\bf (H5)] $\ffi\in \C^1_b(\Rnu^n\to\Rnu^m)$.   
  \item[\bf (H6)]  For all $(t,x,y_1, z_1, y_2, z_2)$, 
   \aa{
  |\pl_x g(t,x,y_1,z_1)  - \pl_xg(t,x,y_2,z_2)|
  \lt K(|y_1-y_2| + |z_1 - z_2|),
  }
  where $K$ is a constant.
    \ei
Note that $X_t\in \mathbb D^{1,2}$ (see, e.g., \cite{nualart}).
Furthermore, it is straightforward  to verify (see the proof of Corollary 
 \ref{cor1111})  that (H1)--(H6) imply (A1)--(A6). 
Therefore, according to \cite{Karo}, $Y_t,Z_t\in \mathbb D^{1,2}$ and 
  there is a version of $\{(D_r^k Y_t, D^k_r Z_t), 0\lt r, t \lt T\}$ satisfying the BSDE
  \mm{
D^k_r Y_t=  \nab \ffi(X_T)D^k_r X_T
 +  \int_t^T \big[ \pl_x g(s, X_s, Y_s, Z_s)D^k_r X_s  \\  +\pl_y g(s,X_s,Y_s,  Z_s)D^k_r Y_s  
+\pl_z g(s,X_s,Y_s,  Z_s)D^k_r Z_s\big] ds 
+ \int_t^T  D^k_r Z_s dB_s \quad  \text{if} \;\; t\gt r,
}
and $D_r^k Y_t = 0$, $D^k_r Z_t = 0$, if $t<r$.
Furthermore, $D^k_r X^l_t$'s solve the SDE (see \cite{nualart})
\aaa{
\lb{f-mal-1111}
D^k_rX^l_t = \sg_{lk}(r) + \sum_{j=1}^n\int_r^t \pl_{x_j} f^l(s,X_s) D^k_r X^j_s ds
}
(written component-wise), where $\sg_{lk}$ are the entries of the matrix $\sg$.
Finally, to guarantee upper and lower bounds for components of $D_r Y_t$, we assume
\bi
\item[\bf (H7)] Each component $g^i(t,x,y,z)$ of $g$ depends only on 
$z^i=(z^{i1}, \ldots, z^{in})$ in the last argument. 
\item[\bf (H8)] 
There exist constants $\gm_{ij}\gt 0$ such that $\pl_{x_j} \ffi^i \gt \gm_{ij}$ for 
all $i=1,\ldots, m$, $j=1,\ldots, n$.
\item[\bf (H9)]  
For all $(t,x,y,z)\in [0,T]\x \Rnu^n\x \Rnu^m \x \Rnu^{m\x n}$,
 $\pl_{y_j}g^i(t,x,y,z) \gt 0$ for $i\ne j$,  $\pl_{x_j} f^i(t,x)\gt 0$ for $i\ne j$, and
 $\pl_{x_j} g^i(t, x,y, z)\gt 0$ for all $i,j$.
 \item[\bf (H10)] $\sg_{ij}(r)\gt 0$ for all $i,j$ and
  $\sum_{j=1}^n \gm_{ij}^2 \int_0^t|\sg_j(r)|^2 dr >0$ for all $i$ and $t\in (0,T]$,
 where $\sg_j = (\sg_{j1},\ldots, \sg_{jn})$ is the $j$-th line of the matrix $\sg$.
 \ei
\begin{cor}
\lb{cor1111}
Assume (H1)-(H10) and let $Y_t$ be the second component of the
unique $\mc F_t$-adapted solution $(X_t,Y_t,Z_t)$ to the FBSDE \rf{bsde}--\rf{fsde1}
 (whose existence is known under (H1)--(H3)). Then, the conclusion
 of Theorem \ref{t-main1111} holds for $Y_t$.
 \end{cor}
\subsection{Gaussian-type bounds may not exist in the absence of (A7) or (H7)}
Consider the system of BSDEs \rf{exl-1111} 
which does not satisfy (A7), but satisfies the rest of
the assumptions.
Below, we show that (i) the Malliavin derivative $D_r Y^1_t$ 
does not take only positive or only negative
values  (as $\om$ varies) for a large interval of values of
$t\in (0,T)$; (ii) at some point $t\in (0,T)$,
$Y^1_t$ possesses a density
with respect to Lebesgue measure, but its density does not have Gaussian-type bounds.
Our examples show that the generator satisfying (A7) or (H7) is the most general one
for which the results for one-dimensional BSDEs (e.g.\cite{aboura, TM, OS}) can be extended
to the multidimensional case.

\paragraph{\it Example 1. $D_r Y^1_t$ takes positive and negative values.}
Consider the BSDE \rf{exl-1111} and let $T\gt 1$.
The BSDE for Malliavin derivatives \rf{mal-1111} for $t\gt r$ takes the form
\eqq{
D_rY^1_t = e^{-T} - \int_t^T D_r Z^2_s ds + \int_t^T D_r Z^1_s dB_s,\\
D_r Y^2_t = 2 + \sin B_T  + \int_t^T D_r Z^2_s dB_s.
}
In what follows, we will make use of the explicit expressions for
the conditional expectations $\E[\cos B_T\, \big| \, \mc F_t]$
and  $\E[\sin B_T\, \big| \, \mc F_t]$.  Applying It\^o's formula to
$e^{\frac{T}2} \cos B_T$ and $e^{\frac{T}2} \sin B_T$ (regarding $T$ as the time variable), 
and noticing that two terms in the resulting expressions cancel each other, we obtain 
\aa{
e^{\frac{T}2}\cos B_T = 1 - \int_0^T e^{\frac{s}2} \sin B_s dB_s; \qquad
e^{\frac{T}2}\sin B_T =  \int_0^T e^{\frac{s}2} \cos B_s dB_s.
}
Therefore, $\{e^{\frac{t}2}\cos B_t\}_{t\gt 0}$ and $\{e^{\frac{t}2}\sin B_t\}_{t\gt 0}$
are martingales. This implies
\aa{
\E[\cos B_T\, \big| \, \mc F_t] 
=  e^{\frac{t-T}2}\cos B_t; \qquad
\E[\sin B_T\, \big| \, \mc F_t] =  e^{\frac{t-T}2}\sin B_t.
}
By  Proposition 5.3 in \cite{Karo}, $D_t Y_t$ is a version of $Z_t$. Therefore, a.s.,
\aaa{
\lb{z1111}
&Z^2_t = 2 + \E[\sin B_T\,\big| \, \mc F_t]  = 2+ e^{\frac{t-T}2}\sin B_t, \\
\lb{rep1111}
& D_r Z^2_t = e^{\frac{t-T}2}\cos B_t = \E[\cos B_T\, \big| \, \mc F_t], \quad r\lt t.
}
Taking the conditional expectation $\E[\,\fdot\, \big| \mc F_t]$
in the BSDE for $D_rY^1_t$, by  \rf{rep1111}, we obtain
\aa{
D_rY^1_t = e^{-T} - \E[\cos B_T\, \big| \, \mc F_t](T-t) = 
 e^{-\frac{T}2}\big(e^{-\frac{T}2} - (T-t)e^{\frac{t}2}\cos B_t\big).
}
The above formula shows that for $t\in (0, T-e^{-\frac{T}2})$,
 $D_rY^1_t$ takes negative and positive values as $\om$ varies;
in particular, $\E[D_rY^1_t] = e^{-\frac{T}2}\big(e^{-\frac{T}2} - (T-t)\big)$ is negative.


\paragraph{\it Example 2. Absence of Gaussian-type bounds on the density of $Y^1_t$.}
Consider again the BSDE \rf{exl-1111} with $T\gt 1$.
By \rf{z1111}, 
\mm{
Y^1_t = e^{-T}B_t - (2 + \E[\sin B_T\, \big| \, \mc F_t])(T-t) \\
=  - 2(T-t)+  e^{-T}\big(B_t- e^{\frac{t+T}2}(T-t)\sin B_t\big).
}
There exists $\tau\in (0,T)$ such that $T-\tau =  e^{-\frac{\tau+T}2}$. We have
\aa{
Y^1_\tau = \al + \beta(B_\tau - \sin B_\tau), \quad \text{where} \; \;
\al = -2(T-\tau), \; \; \beta = e^{-T}. 
}
We show that the density of
$Y^1_\tau$ does not possess Gaussian-type bounds.
Note that the function $x\mto \al + \beta(x - \sin x)$ is strictly increasing.
Let $\psi$ be its inverse function and
 $p_\tau(x) = (2\pi\tau)^{-\frac12} \exp\{-\frac{x^2}{2\tau}\}$ be the Gaussian density.
 For any bounded measurable function $\ffi$, we have
\aa{
\int_{-\infty}^{+\infty} \ffi( \al + \beta(x - \sin x)) p_\tau(x) dx = 
\int_{-\infty}^{+\infty} \ffi(y) \frac{p_\tau(\psi(y))}{\beta(1- \cos \psi(y))} \, dy,
}
where the integration is understood in the Lebesgue sense and
the limits on the right-hand side can be computed from the representation
$\psi(y) = \frac{y}{\beta} - \frac{\al}{\beta} + \sin{\psi(y)}$.
Thus, the density of $Y^1_\tau$ is the function 
\aaa{
\lb{d1111}
y\mto  \frac{p_\tau(\psi(y))}{\beta(1- \cos \psi(y))}. 
}
Note that $p_\tau(\psi(y))$ possesses Gaussian-type bounds. 
However, by the above representation for $\psi$, the denominator in \rf{d1111} takes
a countable number of zero values as $y\to \pm \infty$ 
since $\psi(y)$ reaches all the values $2\pi n$, $n\in \Nnu$. 
Therefore, the density of 
$Y^1_\tau$ does not possess Gaussian bounds.

\subsection{Proofs of Theorem \ref{t-main1111} and Corollary \ref{cor1111}}
\begin{proof}[Proof of Theorem \ref{t-main1111}]
In order to estimate the density $\rho_{Y^i_t}$ by Corollary \ref{cor-main}, 
we have to prove the existence of 
non-negative lower bounds on $D^k_r Y^i_t$  and of an upper bound on $|D_r Y^i_t|$. 
Under  (A7), for each fixed $k$, the BSDE \rf{mal-1111} transforms to
\mmm{
\lb{m-bsde1111}
D^k_r Y^i_t=  D^k_r \xi^i +  \int_t^T \big[ D^k_r g^i(s, Y_s, Z_s)   + 
\sum_{j=1}^n \pl_{y^j} g^i(s,Y_s,  Z_s)D^k_r Y^j_s  \\
+\sum_{j=1}^n \pl_{z^{ij}} g^i(s,Y_s,  Z_s)D^k_r Z^{ij}_s\big] ds 
+ \sum_{j=1}^n \int_t^T  D^k_r Z^{ij}_s dB^j_s.
}
Note that the coefficients of \rf{m-bsde1111} satisfy the assumptions of Theorem \ref{t1111}. Therefore,
for all $i=1,\ldots, m$ and $k=1,\ldots, n$ it holds that
\aaa{
D^k_r Y^i_t \gt \beta_{ik}(r)  \exp\{(T-t) \inf \pl_{y^i} g^i \} = m^{ik}_{r,t},
}
For each $i$, the function $\la_i(t)$ from Corollary \ref{cor-main} can be computed as follows:
\aaa{
\lb{la1111}
\la_i(t) = \int_0^t \sum_{k=1}^n \big(m^{ik}_{r,t}\big)^2 dr = 
  \exp\{2(T-t) \inf \pl_{y^i} g^i  \} \int_0^t \Big(\sum_{k=1}^n \beta_{ik}^2(r)\Big) dr.   
}
By (A8), $\la_i(t)>0$.  
Let us prove now that  for each $k=1,\ldots, n$, $|D^k_r Y_t|$ is bounded from above a.s.
Applying It\^o's formula to 
$|D^k_r Y_t|^2$ and taking the conditional expectation $\E[\,\fdot \, \big| \, \mc F_\tau]$
for some $\tau>0$, for all $t\gt \tau$, from standard estimates and Gronwall's inequality, we obtain
that a.s.
\aa{
\E[\, |D^k_r Y_t|^2 \, \big| \, \mc F_\tau] \lt \Big(\E[\, |D^k_r \xi|^2 \, \big| \, \mc F_\tau] 
+ \int_\tau^T \E[\,|D^k_r g(t,Y_t,Z_t)|^2 \, \big| \, \mc F_\tau] dt\Big) e^{M(T-t)} 
}
for some constant $M>0$. Evaluating the above inequality at $\tau=t$, by (A8) and (A10), we obtain that
a.s.,
\aa{
|D_r Y_t| \lt \sqrt{2}\, \Phi(r) \, e^{\frac{M(T-t)}2}. 
}
Therefore, the function $\La(t)$ from 
 Corollary \ref{cor-main}  is $\La(t) = 2\, e^{M(T-t)} \int_0^t \Phi^2(r) dr$.
 The statement of the theorem follows now from the aforementioned corollary. 
\end{proof}
If $g$ does not depend on $z$ and $\om$, we obtain another bound on $|D^k_r Y_t|$
which will be used for a better estimation of tail probabilities in the gene expression model in Section \ref{gene1111}.
\begin{lem}
\lb{lem1111}
Assume $g$ does not depend on $z$ and $\om$. Then,
under the assumptions of Theorem \ref{t-main1111},
for all $t\in [0,T]$, a.s.,
\aa{
|D^k_r Y_t|   \lt  \E\Big[ \sum_{i=1}^m D^k_r \xi^i \,\big| \, \mc F_t\Big]
\exp\Big\{(T-t)\sup_{j,t,y}\sum_{i=1}^m|\pl_{y_j} g^i(t,y)|\Big\}. 
}
\end{lem}
\begin{proof}
Fix $\tau\in [0,T]$. Summing up the equations in system \rf{m-bsde1111} with respect to $i$
and taking the conditional expectation with respect to $\mc F_\tau$, for all $t\in [\tau,T]$,
we obtain
\mm{
\E\Big[\sum_{i=1}^m D^k_r Y^i_t \,\big| \, \mc F_\tau\Big]
= \E\big[ \sum_{i=1}^m D^k_r \xi^i  ds  \,\big| \, \mc F_\tau\big]\\
+\int_t^T
 \E\Big[\sum_{j=1}^n \Big(\sum_{i=1}^m \pl_{y^j} g^i(s,Y_s) \Big)D^k_r Y^j_s   \,\big| \, \mc F_\tau\Big] ds.
}
By the non-negativity of $D^k_rY^i_t$ and Gronwall's inequality, a.s.,
\aa{
\E\Big[\sum_{i=1}^m D^k_r Y^i_t \,\big| \, \mc F_\tau\Big]  
\lt 
\E\Big[ \sum_{i=1}^m D^k_r \xi^i   \,\big| \, \mc F_\tau\Big]
\exp\Big\{(T-t)\sup_{j,t,y}\sum_{i=1}^m|\pl_{y_j} g^i(t,y)|\Big\}. 
}
Since $|D^k_r Y_t| \lt \sum_{i=1}^m D^k_r Y^i_t$, evaluating 
the both parts of the above estimate at $t=\tau$ concludes the proof.   
\end{proof}
\begin{proof}[Proof of Corollary \ref{cor1111}]
We start by obtaining non-negative lower bounds on $D^k_rX^j_t$.
Consider equation \rf{f-mal-1111} for a fixed $k\in\{1,\ldots, n\}$ and $j$ varying from $1$ to $n$.
By Proposition \ref{pro24}, for all 
$j = 1,\ldots, n$ and for all $t\in (0,T]$,
\aa{
D^k_rX^j_t \gt  \sg_{jk}(r)\exp\big\{\int_r^t \pl_{x_j} f^j(s,X_s) ds\big\} \gt  
\sg_{jk}(r)\exp\{-KT\},
}
where $K= \sup_{[0,T]\x \Rnu^n}|\pl_x f(t,x)|$.
This is valid for all $k=1\ldots, n$.
Also,
\aaa{
\lb{x-bound1111}
|D_r X_t| \lt |\sg(r)| \exp\{K T\}.
}
Since $D^k_rX^j_t$ and $|D_r X_t|$ possess the above-mentioned lower and upper bounds,
 the verification of most of the assumptions of Theorem 
\ref{t-main1111} is straightforward. In particular, since for all $(y,z)$,
\aa{
D^k_r g(t,X_t,y,z) = \sum_{j=1}^n \pl_{x_j} g(t,X_t,y,z) D^k_r X^j_t,
}
where 
 $\pl_x g$ is bounded, then $|D^k_r g(t,X_t,y,z)|\lt C_1 |\sg(r)|$
 for some constant $C_1>0$.
Furthermore, by the boundedness of the partial derivatives of $g$, 
there exists a consatnt $C_2>0$ such that
\aa{
|g(t,x,y,z)| \lt |g(t,0,0,0)| + C_2(|x| + |y| + |z|).
}
The above arguments, together with the fact that $X_t$ possesses moments of
all orders, imply (A1)--(A6) and (A10). Furthermore, (A3) holds because
for all $(y,z)$, $(t,x)\mto \pl_x g(t,x,y,z)$ is a Borel function.
Let us show now that assumption (A8) is fulfilled. 
We have
\aa{
D_r^k \ffi^i(X_T) = \sum_{j=1}^n \pl_{x_j} \ffi^i(X_T) D^k_r X^j_T
\gt e^{-KT} \sum_{j=1}^n \gm_{ij} \sg_{jk}(r) = \beta_{ik}(r),
}
where $\beta_{ik}(r)$ are defined by the last expression. By (H10),
\aa{
\int_0^t \Big(\sum_{k=1}^n \beta_{ik}^2(r)\Big) dr \gt e^{-2KT} 
 \sum_{j=1}^n\gm^2_{ij} \int_0^t  \Big(\sum_{k=1}^n \sg_{jk}^2(r)\Big) dr >0.
}
On the other hand,
\aa{
|D_r\ffi(X_T)| \lt |\nab \ffi(X_T)| |D_r X_T|\lt  |\sg(r)| \sup_x |\nab \ffi(x)| e^{K T}.
}
This implies (A8).
\end{proof}
\begin{rem}
\rm
If the forward SDE is one-dimensional, then
the diffusion coefficient $\sg$ may also depend on $X_t$, i.e.,
the SDE can take the form
$X_t = x + \int_0^t f(s,X_s) ds + \int_0^t \sg(s,X_s) dB_s$. In this case,
deterministic non-negative lower and upper bounds for
$D_r X_t$ can be obtained   
by means of the Lamperti tranform (see, e.g., \cite{OS}).
\end{rem}
\section{Application to gene expression}
\lb{gene1111}
In \cite{gene}, the authors proposed a BSDE approach to model protein level dynamics
for a gene regulatory network. Distributions of proteins, generated by the genes of the network, 
were represented in the form of histograms which resembled Gaussian-type densities. 
Here, we aim to prove that
under certain assumptions on the parameters of the model, the distributions of proteins
indeed possess densities with respect to Lebesgue measure. Moreover, we will use the results of the previous section to obtain upper and lower Gaussian-type bounds  on these densities.
Furthermore, we compute
prediction intervals for expressions of individual genes
 and show the positivity of amounts of proteins.
Finally, we demonstrate how Gaussian-type bounds can be used as a validation tool of the gene expression
model. 
In Subsection \ref{precise1111}, we present a reasonable biological model 
where upper and lower density bound curves are relatively close to each other
and show that these bounds agree with the density profile  of benchmark data obtained
by Gillespie's stochastic simulation algorithm (SSA) for modeling of biochemical reactions \cite{Gil}.
\subsection{Brief description of the gene expression model}
Let $\eta_t = (\eta^1_t, \ldots, \eta^n_t)$ denote a vector whose $i$-th component 
is the amount of protein generated by gene $i$. According to \cite{gene}, the dynamics of $\eta_t$ is described
by the BSDE
\aaa{
\lb{bsde2}
\eta_t = \eta_T - \int_t^T f(\eta_s)\, ds -  \int_t^T z_s\, dB_s, \;
\; t\in [0,T].
}
In \rf{bsde2}, 
$B_t$ is an $n$-dimensional Brownian motion,  
$\eta_T$ is a given final data (obtained through a simulation using Gillespie's SSA), and
the $i$-th component of $f$ is the synthesis/degradation rate
of the $i$-th gene which takes the form
\aaa{
\lb{sdrate}
f^i(\eta) = \frac{\nu_i}{1+\exp(-\Te_i)} - \rho_i \eta^i,
}
where $\Te_i =\sum_{j=1}^n A_{ij} \eta^j$
has the meaning of the total regulatory input to gene $i$ by other genes of the network.
In particular,
$A_{ij}\eta^j$ represents the regulatory effect of gene $j$ to gene $i$ with
$A_{ij}$ being the strength of this regulation.
Each element $A_{ij}$  can be
negative, positive, or equal zero, indicating repression, activation, or
non-regulation, respectively, of gene $i$ by gene $j$. Furthermore,
$\nu_i$ denotes the maximum synthesis rate of the $i$-th protein, while the same
protein degrades at rate $\rho_i\eta^i$. The final condition $\eta_T$, determined through a simulation
by using Gillespie's SSA, is  represented in the form
$h(B_T)$, where the function $h:\Rnu^n\to\Rnu^n$ was found
in such a way that $h(B_T)$ matches $\eta_T$.
According to \cite{gene}, $h$ looks like a linear function.

The BSDE \rf{bsde2} was solved by means of the associated final value problem for the PDE
\aaa{
\lb{pde-model}
\begin{cases}
\pl_t \te(t,x) + \frac12\lap_x\te(t,x) - f(\te(t,x)) = 0,\\
\te(T,x) = h(x), \quad x\in \Rnu^n,
\end{cases}
}
and the solution $\eta_t$ to \rf{bsde2}  was computed as
$\te(t,B_t)$ by means of generating multiple Brownian motion paths. 

\subsection{Bounds on the density and tail probabilities}
\lb{62-1111}
In what follows, we show how the results obtained in the previous section can be applied to
estimate the protein level of a gene in a gene regulatory network. 
\begin{thm}
\lb{thm1111}
Let $\eta_t$ be the solution to the BSDE \rf{bsde2},
with the final data $\eta_T = h(B_T)$. Assume that the partial derivatives of
$h$ satisfy the condition 
$\gm_{ik} \lt \pl_{x_k} h^i \lt \Gm_{ik}$, where $\Gm_{ik}$ and $\gm_{ik}\gt 0$ are constants
with the property $\max_k \gm_{ik}>0$ for each $i$. 
Further assume that the rate function $f$ is given by \rf{sdrate} with $A_{ij}\lt 0$ for $i\ne j$.
Then, each component $\eta^i_t$ of $\eta$
has a density $\rho_{\eta^i_t}$ w.r.t. Lebesgue measure.
Moreover, etimates \rf{est1111} and \rf{est2222} hold with $Y^i_t = \eta^i_t$
and $\la_i(t)$, $\La(t)$ given by the expressions
\aaa{
\lb{lai1111}
&\la_i(t) = \begin{cases}
 |\gm_i|^2 t\exp\{2(T-t)\rho_i\} \; \; \text{if} \; A_{ii}\lt 0,\\
 |\gm_i|^2 t\exp\!\big\{2(T-t)\big(\rho_i- \frac{A_{ii}\nu_i}4\big)\big\}   
 \;\; \text{if} \; A_{ii} > 0,
 \end{cases}  \\
 \lb{La1111}
 &\La(t) = t \sum_{k=1}^n \Big(\sum_{i=1}^n \Gm_{ik}\Big)^2
\exp\{2(T-t) \max_j P_j\},
}
where $\gm_i = (\gm_{i1},\ldots, \gm_{in})$ and
\aaa{
\lb{Pj1111}
P_j = \sum_{i\ne j} \frac{\nu_i|A_{ij}|}4 + \max\{\rho_j, |\sfrac{A_{jj}\nu_j}4 - \rho_j|\}.
}
\end{thm}
 \begin{proof}
 First, we note that 
 \aaa{
 \lb{psi1111}
 \pl_{\eta^j} f^i(\eta) = \psi(\Te_i) \nu_i A_{ij} - \rho_i \dl_{ij}, \;  \text{where} \; \;
 \psi(x) = \Big(\frac1{1+e^{-x}}\Big)' \! = \frac1{2(\cosh(x) + 1)}.
 }
 Furthermore, for all $r\in [0,T]$,
 \aa{
 D^k_r h^i(B_T) = \pl_{x_k} h^i(B_T) \quad \text{and}
 \quad D_rh(B_T) = \nab h(B_T).
 }
We are in the assumptions of Theorem \ref{t-main1111}, so we conclude that
each component $\eta^i$ has the density w.r.t. Lebesgue measure possessing
upper and lower Gaussian-type bounds. 
Let us compute functions $\la_i(t)$ and $\La(t)$ from Theorem \ref{t-main1111}.
By \rf{la1111},
\aa{
 \la_i(t) =  |\gm_i|^2 t \exp\{-2(T-t) \sup \pl_{\eta^i} f^i(\eta) \} 
}
which implies \rf{lai1111}.
Furthermore, by Lemma \ref{lem1111},
\mm{
|D^k_r \eta_t|   \lt  \E\Big[ \sum_{i=1}^n \pl_{x_k}h^i(B_T)   \,\Big| \, \mc F_t\Big]
\exp\Big\{(T-t)\sup_{j,\eta}\sum_{i=1}^m|\pl_{\eta^j} f^i(\eta)|\Big\} \\  \lt
\Big(\sum_{i=1}^n \Gm_{ik}\Big) 
\exp\Big\{(T-t) \max_j P_j\Big\}, 
}
where $P_j$ is given by \rf{Pj1111}. Recall that 
$\La(t) = \int_0^t M_{r,t}^2 dr$,
where $M_{r,t}$ is an upper bound for $|D_r \eta_t|$ (see Corollary \ref{cor-main}). 
This immediately implies \rf{La1111}.
\end{proof}
The following corollary is useful for numerical computations in the gene expression model
in Subsection \ref{numsec}.
\begin{cor}
Let the rate function for the gene expression model is as in Theorem \ref{thm1111}.
Further let the final data are Gaussian and given by
$h^i_T(B_T) = c_iB^i_T + b_i$, where $c_i,b_i>0$, $i=1,\ldots, n$. Then, the function  
$\La(t)$ is computed as follows:
\aa{
\La(t) = t  \exp\{2(T-t) \max_j P_j\} \sum_{i=1}^n c_i^2,
}
where $P_j$ is given by \rf{Pj1111}.
\end{cor}

\subsubsection{Prediction intervals and positivity of $\eta^i_t$}
\lb{method1111}
We have the following corollary of Theorem \ref{thm1111}.
\begin{cor}
\lb{cor-pred-1111}
Under the assumptions of Theorem \ref{thm1111}, the following estimate
holds for the $\al\! \cdot \!100\%$ prediction interval, where $0<\al<1\!:$
\aaa{
\lb{pred1111}
\PP \big\{\eta^i_t \in \big(\E\eta^i_t - x_\al, \E\eta^i_t+ x_\al\big)\big\}\gt \al,
}
where $x_\al = \sqrt{2\La(t) \ln\Big(\frac2{1-\al}\Big)}$ and $\La(t)$ is given by
\rf{La1111}.
\end{cor}
\begin{proof}
Using formulas \rf{est2222} for tail probabilities, we can evaluate the probability for a prediction interval
from above. Namely, for all $x>0$, 
\mm{     
\PP \big\{\eta^i_t \in \big(\E\eta^i_t - x, \E\eta^i_t+ x\big)\big\} =
1- \PP \big\{\eta^i_t \gt \E\eta^i_t+x\big\} - \PP \big\{\eta^i_t \lt \E\eta^i_t-x\big\}\\
\gt 1- 2 \exp\Big(-\frac{x^2}{2\Lambda(t)}\Big) = \al.
}
Expressing $x_\al$ from the last equation, we obtain \rf{pred1111}.
\end{proof}
\begin{rem}
\rm
In gene expression models, we usually deal with an ensemble of cells. 
In our model, each cell contains a gene that expresses the $i$-th protein.
The result of Corollary \ref{cor-pred-1111} implies that for at least $\al \cdot 100\%$ of cells, 
the amount of the $i$-th protein, which we denote by $\eta^i_t$, lies inside the interval
$\big(\E\eta^i_t - x_\al, \E\eta^i_t+ x_\al\big)$.
\end{rem}
One can use the second inequality in \rf{est2222} for showing the positivity of $\eta^i_t$'s.
\begin{cor}
\lb{cor+1111}
Under the assumptions of Theorem \ref{thm1111}, 
\aaa{
\lb{+1111}
\PP \{\eta^i_t \lt 0\} \lt \exp\Big(-\frac{ (\E \eta^i_t )^2}{2\La(t)}\Big).
}
\end{cor}
Note that 
it is not possible to prove that $\eta^i_t$'s are always positive because
this would mean that their densities do not have Gaussian-type bounds. The only
way to show this positivity (for the given parameters of the model)
 is to compute the probability on the right-hand side of \rf{+1111} and to see whether
 it is negligibly small. 

See Subsections \ref{pr1111}  and \ref{positivity1111}
for numerical results related to this subsection. 
\subsubsection{Precise Gaussian-type bounds for a simplified model}
\lb{precise1111}
In general, curves representing bounds for the density are not 
close to each other, so there is not too much information about the real density curve. 
This happens because obtaining density  bounds is related to obtaining upper and lower
bounds for the function $\ffi_F$ defined by \rf{ffiF}.
In some cases, however, it is possible to find more precise Gaussian-type density bounds.
This, in particular, allows to see how experimental data (in our case, numerical data obtained
by Gillespie's method) fit between the curves. 

We describe a gene regulatory network where
we aim to suppress expression of a certain gene, say gene 1.
The need  to lower expression of specific genes may arise in disease treatment,
such as cancer, neurodegenerative diseases, or viral infections \cite{cox, ghosh, lee, li, rinaldi}.
Suppression of gene expression, known as gene silencing or gene knockout, 
in practice can be achieved by an antisense therapy \cite{miller, rinaldi} (when single-stranded short synthetic DNA molecules are delivered inside the cell), 
genomic editing \cite{cox, li}, or administration of antibodies targeting virus gene expression \cite{lee}.

We model suppression of gene expression by introducing a gene regulatory network, where
 genes $2, \ldots, n$ repress gene 1, while the latter does not regulate the other genes. This means that $A_{1i}<0$ for
all $i=2, \ldots, n$. In addition, we assume that each gene activates itself and that genes $2, \ldots, n$
do not regulate each other. This implies that $A_{ii}>0$ for all $i$ and $A_{ij} = 0$ if $i\ne j$ and $i\ne 1$.

By Theorem \ref{thm1111}, for the model described above, 
one can find bounds on the density of the distribution of the first (targeted) protein. Furthermore,
for a sufficiently large class of the parameters of the model, 
one can reasonably estimate the density profile of gene 1.
See Subsection \ref{plots1111} 
for numerical results, diagrams, and figures
based on the results of this subsection. 
\begin{thm}
\lb{thm5}
Let the coefficients $A_{ij}$ and the rate function $f$ be as described above.
Then, the first component $\eta^1_t$ of the solution $\eta_t$ to the BSDE \rf{bsde2}
with the final data $\eta^i_T = c_iB^i_T+b_i$, $i=1,\ldots, n$, where $c_i,b_i>0$, has a density $\rho_{\eta^1_t}$ w.r.t. Lebesgue measure.
Moreover, it holds that
\aa{     
\frac{\E |\eta^1_t-\E \eta^1_t|}{2t M_t^2} \exp\Big(\!\!-\frac{\big(x-\E \eta^1_t\big)^2}{2tm_t^2}\Big) \lt 
\rho_{\eta^1_t}(x) \lt \frac{\E |\eta^1_t-\E \eta^1_t|}{2tm_t^2} \exp\Big(\!\!-\frac{\big(x-\E \eta^1_t\big)^2}{2tM_t^2}\Big),
}
where 
\aa{
m_t =  c_1 e^{(\rho_1-\frac{\nu_1 A_{11}}4)(T-t)},  \quad
M_t = e^{\rho_1(T-t)} \sqrt{c_1^2 + \sum_{k=2}^n(\kappa^k_t)^2}
 }
with  $\kappa^k_t = \frac{\nu_1 |A_{1k}| c_k}4 e^{\rho_k(T-t)}(T-t)$.
 
\end{thm}
\begin{proof}

As it was discussed in Subsection \ref{sec1111}, the BSDE for $D^k_r \eta^1_t$ takes the form
\mmm{
\lb{x1e1}
D^k_r \eta^1_t = \dl_{k1}c_1 - \int_t^T \nu_1 \psi(\Te_1) \sum_{j\ne 1} A_{1j} D^k_r \eta^j_s \, ds
\\ + \int_t^T(\rho_1 - \nu_1 \psi(\Te_1) A_{11}) D^k_r \eta^1_s\, ds  + \int_t^T   D^k_r z^1_s dB_s,
}
where $\psi$ is defined in \rf{psi1111}. 
We claim that if $k = 1$, the second term on the right-hand side of \rf{x1e1} equals zero.
Indeed, for $i\ne 1$,
\aa{
D^k_r \eta^i_t  = \dl_{ki}c_i + 
\int_t^T(\rho_i - \nu_i \psi(\Te_i) A_{ik}) D^k_r \eta^i_s\, ds  + \int_t^T   D^k_r z^i_s dB_s.
}
By the equation in \rf{u1111} (see also Corollary \ref{cor-qi1111}),
\aaa{
\lb{kk1111}
D^k_r \eta^i_t  = \E\big[ \dl_{ki}c_i e^{\int_t^T (\rho_i - \nu_i \psi(\Te_i) A_{ik}) ds}\big | \mc F_t\big]
}
which, in particular, implies that $D^k_r \eta^i_t = 0$ if $k\ne i$ and $i\ne 1$. Therefore,
for $k=1$, the second term on the right-hand side of \rf{x1e1} equals zero. This implies that
\aa{
D^1_r \eta^1_t =  \E\big[ c_1 e^{\int_t^T (\rho_1 - \nu_1 \psi(\Te_1) A_{11}) ds}\big | \mc F_t\big].
}
Hence, we have the estimate
\aaa{
\lb{e78}
c_1 e^{(\rho_1-\frac{\nu_1 A_{11}}4)(T-t)} \lt D^1_r \eta^1_t  \lt c_1 e^{\rho_1(T-t)}.
}
Next, if $k\ne 1$, by \rf{u1111} and \rf{x1e1},
\aa{
D^k_r \eta^1_t =  \E\Big[  
 \int_t^T   e^{\int_t^s  (\rho_1 - \nu_1 \psi(\Te_1) A_{11}) dr} 
\nu_1 \psi(\Te_1) |A_{1k}| D^k_r \eta^k_s ds \,\big|\,\mc F_t \Big].
}
Since, by \rf{kk1111}, $0< D^k_r \eta^k_t \lt c_k e^{\rho_k(T-t)}$, we obtain the estimate
\aa{
0< D^k_r \eta^1_t \lt 
 \frac{\nu_1 c_k\, |A_{1k}|}4(T-t) \, e^{(\rho_1+\rho_k)(T-t)}.
}
Together with  inequalities \rf{e78} and Corollary \ref{cor-main}, the above estimate implies the statement of the theorem.
\end{proof}
\subsection{Numerical results}
\lb{numsec}
In this subsection, we compute prediction intervals, the probabilities $\PP\{\eta^i_t \lt 0\}$
(which turn out to be neglidibly small), and obtain
numerical bounds on the density $\rho_{\eta^1_t}$ by Theorem \ref{thm5}.

\subsubsection{Computation of the expectations $\E \eta^i_t $ and $\E|\eta^i_t - \E \eta^i_t|$}
The expectations $\E \eta^i_t $ and $\E|\eta^i_t - \E\eta^i_t|$
are computed by the formulas
\eqn{
\lb{exps}
&\E \eta^i_t  = \frac1{(2\pi t)^{\frac{n}2}}\int_{\Rnu^n} \te^i(t,x) \, e^{-\frac{|x|^2}{2t}} dx,\\
&\E\big|\eta^i_t  - \E \eta^i_t \big|
=\frac1{(2\pi t)^{\frac{n}2}}\int_{\Rnu^n} |\te^i(t,x) - \E \eta^i_t |\, e^{-\frac{|x|^2}{2t}} dx,
}
where $\te^i(t,x)$ is the $i$-th component of the solution $\te(t,x)$ to the final value problem \rf{pde-model} in which the $i$-th component of the final condition $h(x)$
takes the form $h^i(x) = c_i x_i + b_i$. Above, $c_i$ and $b_i$ are positive constants 
obtained in such a way that 
the distribution of $h(B_T)$ coincides with the distribution of $\eta_T$, where the latter is generated
by Gillespie's method. 
We aim to find a numerical solution to problem \rf{pde-model} (at time $t$) 
only in the cube $Q_a = \{x\in\Rnu^n, |x_i|\lt a\}$  chosen in such a way that in the integrals \rf{exps}, the integration
over $\Rnu^n$ can be replaced with the integration over $Q_a$ while keeping the  computational  error small.
Thus, we have to decide how we choose $a$ and how we obtain a numerical solution to problem \rf{pde-model} in $Q_a$.

We start by evaluating the parameter $a$ based on the following two conditions: 
 1) we make sure that the random variable $B_T$, when simulated, takes values 
in $Q_a$; 2) we evaluate the error  in computing expectations \rf{exps} which comes from substituting
the actual area of integration by $Q_a$. The probability that $B_T$ is in $Q_a$ can be easily computed. On the other hand, the parameter $a$ 
can be evaluated by visualizing the simulation of $B_T$ because in practice there are no values of $B_T$ outside of some compact region.
Thus, we have to confirm that the error in computing the expectations \rf{exps} is small. 
Note that if $\tet_t^{\tau,x}$ is the solution to the BSDE \rf{bsde2}  
with the final condition $h(x+B_T-B_\tau)$, it holds that $\tet^{\tau,x}_\tau = \te(\tau,x)$ 
for all $\tau\in [0,T]$. We can evaluate $\E[\tet^{\tau,x,i}_t]$ from the 
associated BSDE.
Indeed, it follows that
\aa{
\E[\tet^{\tau,x,i}_t] = e^{\rho_i(T-t)}\E[h^i(x+B_T-B_\tau)] - \int_t^T e^{\rho_i(s-t)} \E\Big[\frac{\nu_i}{1+e^{-\Te_i}}\Big]\, ds,
}
where $i$ in the upper index stands for the $i$-th component. Therefore,
\aa{
\te^i(\tau,x) = e^{\rho_i(T-\tau)}(c_i x_i + b_i) - \int_\tau^T e^{\rho_i(s-\tau)} \E\Big[\frac{\nu_i}{1+e^{-\Te_i}}\Big]\, ds.
}
This implies that
\aa{
|\te^i(t,x)| \lt e^{\rho_i(T-t)}\big(c_i|x_i|+b_i + \nu_i(T-t)\big).
}
By the first equation in \rf{exps}, we obtain  a  bound  on ${\rm Err} \, \E \eta^i_t$ by 
computing the integral of the right-hand side of the above estimate
over ${\Rnu^n \dd Q_a}$:
\mmm{
\lb{abexp}
{\rm Err} \, \E \eta^i_t \lt  \frac1{(2\pi t)^{\frac{n}2}}\int_{\Rnu^n \dd Q_a}|\te^i(t,x)| \, e^{-\frac{|x|^2}{2t}} dx\\
\lt  e^{\rho_i(T-t)} J(t,a)\big(ac_i + c_i(n-1) \sqrt{2t \pi^{-1}} + b_in + \nu_i n(T-t)\big),
}
where 
$J(t,a) = \sqrt{\frac{2t}{\pi a^2}} \, e^{-\frac{a^2}{2t}}$. Above, we used the estimate
$\int_z^{+\infty} e^{-x^2} dx  < \frac1{2z}\, e^{-z^2}$.
Finally, we have
${\rm Err} \, \E\big|\eta^i_t  - \E[\eta^i_t ]\big| \lt 2\, {\rm Err \,}  \, \E |\eta^i_t|$, 
where $\E|\eta^i_t|$ is also estimated by the right-hand side of \rf{abexp}.

To obtain a numerical solution to problem \rf{pde-model} in $Q_a$, we solve the PDE \rf{pde-model} in a larger region, namely in
the cube $Q_{a+N}$, while using the boundary condition $\frac{\pl}{\pl {\mb n}} \te(t,x) = 0$, where
$\mb n$ is the unit normal vector to the boundary of this cube. This problem is equivalent to the final
 value problem of  the form \rf{pde-model} with the final condition
\aa{
h^i_N(x) =
\begin{cases}
c_ix_i + b_i \; \text{if} \; |x_i|\lt a+N,\\
c_i(a+N)  + b_i \; \text{if} \; x_i > a+ N,\\
c_i(-a-N) + b_i  \; \text{if} \; x_i < -a- N.
\end{cases}
}
We choose $N$ in such a way that the solutions
$\te(t,x)$ and $\te_N(t,x)$ to the final value problem \rf{pde-model} with the final conditions
$h(x)$ and $h_N(x)$ are close enough within $Q_a$. To evaluate this difference, 
we again use the associated BSDE.
By the standard arguments,
 \aa{
 |\te(t,x) - \te_N(t,x)| \lt e^{M(T-\tau)}\Big( \E|h(x+B_T - B_\tau) - h_N(x+B_T - B_\tau)|^2\Big)^\frac12,
 }
where $M$ is a bound on 
 $\nab f$, which can be computed using expression \rf{sdrate}. For this bound, one can take, for example, $M = \sqrt{\sum_i 2m_i^2}$,
 where  $m_i = \max\{\rho_i, \frac{\nu_i |A_i|}4\}$ and $|A_i| = \sqrt{\sum_{j=1}^n A_{ij}^2}$.
Using the explicit form of $h$ and $h_N$, we compute the right-hand side of the above inequality, which gives
\aa{
\sup_{x\in Q_a} |\te(\tau,x) - \te_N(\tau,x)| 
\lt \frac{|c|  (T-\tau)^\frac34}{N^\frac12} \, e^{M(T-\tau)-\frac{N^2}{4(T-\tau)}},
}
where $c=(c_1, \ldots, c_n)$.
Thus, we  consider $\te_N(t,x)$ as a numerical solution to problem \rf{pde-model} in $Q_a$.
To obtain $\te_N$, problem \rf{pde-model} was transformed to an initial problem by the 
time change $t \leftrightarrow T-t$.
The resulting system of PDEs  with the initial condition $h_N(x)$ and the
boundary condition $\frac{\pl}{\pl{\mb  n}}\, \te = 0$ was solved by the fractional step method \cite{Yan}
employing the Crank-Nicolson scheme in each spatial direction.

\subsubsection{Computation of prediction intervals}
\lb{pr1111}
We computed prediction intervals using the formulas
obtained in Corollary \ref{cor-pred-1111}.
We considered a network of three fully interacting genes with the final data 
$h^i(B_T) = c_i B^i_T + b_i$, $i=1,2,3$ and
the following set of parameters:
 $T=6$, $t=3$, $c_1 = 5$, $c_2 = 0.5$, $c_3 = 0.3$, $b_1 = 150$, $b_2 = 70$, $b_3 = 80$,
 $\nu_1 = 0.5$, $\nu_2 = 0.75$, $\nu_3 = 1.0$, $\rho_1 = 0.2$, $\rho_2 = 0.5$,
 $\rho_3 = 0.6$, and the matrix $A= \{A_{ij}\}$ given by
\aa{
A =
 \begin{pmatrix}
 2.5  &-0.2  &-0.25\\
 -0.03  &3.0  &-0.3\\
  -0.5  &-0.1   &2.0
 \end{pmatrix}.
 }
We obtained 
\aa{
&\PP\{\eta^1_t \in (129.8, 409.5)\} \gt 0.75;  &&\PP\{\eta^1_t \in (83.4, 455.9)\} \gt 0.95;\\
&\PP\{\eta^2_t \in (157.7, 437.4)\} \gt 0.75;   &&\PP\{\eta^2_t \in (111.3, 483.8)\} \gt 0.95;\\
&\PP\{\eta^3_t \in (311.7, 591.4)\} \gt 0.75;   &&\PP\{\eta^3_t \in  (265.3, 637.8)\} \gt 0.95.
}

\subsubsection{Positivity of $\eta^i_t$}
\lb{positivity1111}
In the situation described in the previous subsection, in particular, using the same set
of parameters, we obtained the estimates on probabilities of the events that the components of $\eta_t$
are non-positive:
\aa{
\PP(\eta^1_t \lt 0) \lt  4\cdot 10^{-4}; \quad \PP(\eta^2_t \lt 0) \lt  8\cdot 10^{-5}; 
\quad \PP(\eta^3_t \lt 0) \lt  4\cdot 10^{-10}.
}
This confirms that each $\eta^i_t$ is positive.
In \cite{gene}, 
the positivity of $\eta^i_t$ is a result of agreeing of the BSDE method with Gillespie's  SSA and the fact that the final condition $\eta^i_T$, provided by SSA,
was modeled as $c_i B^i_T + b_i$ 
by fitting the parameters $c_i$ and $b_i$. Here, for the same purpose, we use the result of
Subsection \ref{method1111}. 
\subsubsection{Numerical bounds on the density in the simplified model}
\lb{plots1111}
For the model described in Subsection \ref{precise1111}, we
performed two simulations with different sets of parameters; 
the number of genes in both simulations
was taken three.
The first simulation was performed with the following parameters:
$\rho_1 = \rho_2 =\rho_3 = 1$; $\nu_1 = 0.4$, $\nu_2 = 0.1$, $\nu_3 = 0.3$;
$c_1 = 4.89$, $c_2 = 0.47$, $c_3 = 0.51$, $b_1 = 75.98$, $b_2 = 7.84$, $b_3 = 8.85$;
$A_{11} = 0.1$, $A_{22}=0.04$, $A_{33} = 0.6$, $A_{12} = A_{13} = -2$, $A_{ij} = 0$ if $i\ne j$ and $i\ne 1$;
$T = 4$ and $t=2$. Here, the role of the second and the third genes in  repressing the first gene
is not so significant.  The self-degradation of the first gene plays a bigger role compared to the other simulation.
The second simulation was performed with the parameters:
$\rho_1 = 0.05$, $\rho_2 =\rho_3 = 10^{-4}$; $\nu_1 = 5$, $\nu_2 = 1$, $\nu_3 = 1$;
$c_1 = 1$, $c_2 = 10^{-2}$, $c_3 = 10^{-2}$, $b_1 = 55.32$, $b_2 = 712.34$, $b_3 = 834.02$;
$A_{11} = 10^{-5}$, $A_{22}=10^{-2}$, $A_{33} = 0.1$, $A_{12} = -4$, $A_{13} = -3.5$, $A_{ij} = 0$ if $i\ne j$ and $i\ne 1$;
$T = 18$ and $t=9$. Here, the second and the third genes play a bigger role in repressing the first
gene. This happens because $\Te_1 = A_{11}\eta^1 + A_{12} \eta^2 + A_{13}\eta^3$ is a big negative number 
  reducing the synthesis rate of the first protein  according to formula \rf{sdrate}, and thus, allowing it to degrade.
  
  The density bound curves were computed by Theorem \ref{thm5}; specifically,
  by computing the expressions for $m_t$ and $M_t$, and the expectations $\E\eta^1_t$ and
$\E|\eta^1_t - \E\eta^1_t|$. 
 \vspace{-2.2cm}
 \begin{figure}[h]
 \begin{minipage}{.5\textwidth}
\begin{flushleft}
\begin{overpic}[width=.95\linewidth]{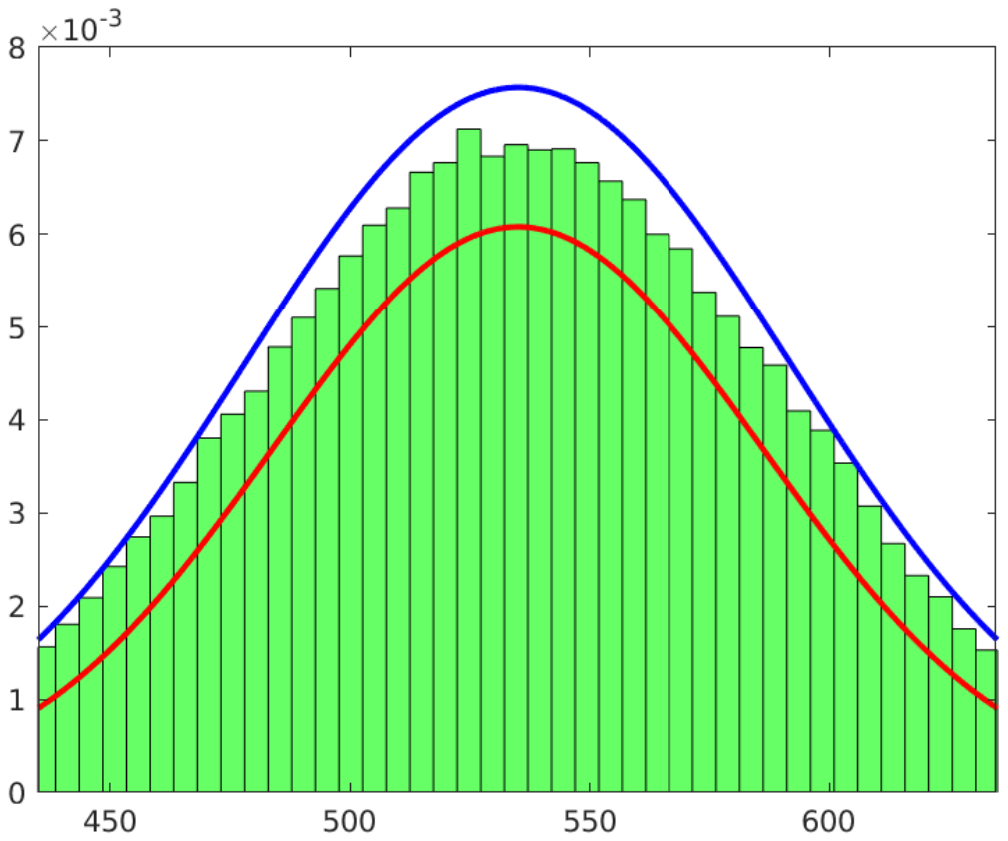}
\put(18,73.5){\tiny{1st simulation, $t=2$, gene 1}}
\put(23,25){\tiny{Number of proteins}}
\end{overpic}
\caption*{}
\end{flushleft}
\end{minipage}%
\begin{minipage}{.5\textwidth}
\begin{flushright}
\begin{overpic}[width=.95\linewidth]{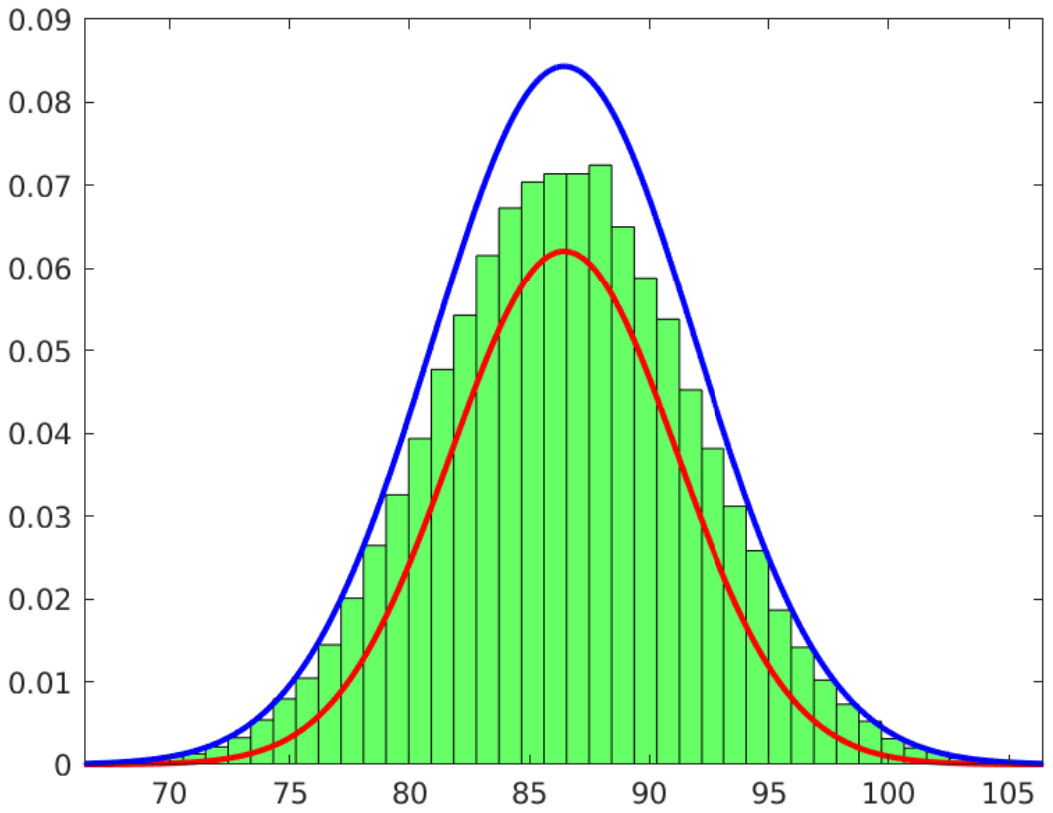}
\put(18,73.5){\tiny{2nd simulation, $t=9$, gene 1}}
\put(23,25){\tiny{Number of proteins}}
\end{overpic}
\caption*{}
\end{flushright}
\end{minipage}%
\end{figure}

Furthermore, we performed a simulation using Gillespie's SSA.
First of all, we did it to obtain the parameters 
$c_i$ and $b_i$, $i=1,2,3$ (in such a way that the distributions of $\eta^i_T$ and $c_i B^i_T + b_i$ coincide), but also, to demonstrate
that our density estimates agree with the data generated by SSA.
To verify the latter, we plotted a histogram (in green, see the figures) obtained by SSA and observed that it fits quite well between the curves
computed by  using Theorem \ref{thm5}.
Remark that obtaining the  aforementioned density bounds is fully based on the BSDE \rf{bsde2} which is in the heart of the BSDE method
introduced in \cite{gene}. By this, the validity of this method is confirmed once again.

\section*{Acknowledgments}
The authors thank the anonymous referee for valuable suggestions towards 
improving the quality of this work.
R.C. acknowledges support from 
 AL ARISE (Ref. LA/P/0112/2020), SYSTEC Base (Ref. UIDB/00147/2020)UIDB/00147/2020), and Programmatic (Ref. UIDP/00147/2020) funding by the Foundation for Science and Technology (FCT), Portugal.
E.S. acknowledges partial support from 
Universidade Federal da Para\'iba through the grant
PROPESQ/PRPG/UFPB/PIA13631-2020.


\begin{thebibliography}{99999}
\bibitem{aboura}
O. Aboura, S. Bourguin,
Density Estimates for Solutions to One Dimensional Backward SDE's. Potential Anal (2013) 38:573--587.
%
\bibitem{antonelli} F. Antonelli, A. Kohatsu-Higa, Densities of one-dimensional backward SDEs. Potential Anal (2005) 22(3):263--287.
\bibitem{cox} D. B. T. Cox, R. J. Platt, F. Zhang, Therapeutic genome editing: prospects
and challenges, Nat Med. 21(2), (2015) pp. 121--131. 
\bibitem{Ci}
J. Cvitanic and J. Ma, Hedging options for a large investor and forward-backward SDEs,
Ann. Appl. Probab., (1996) 6, pp. 370--398.
\bibitem{DelbaenTang}
F. Delbaen, S. Tang,  Harmonic analysis of stochastic equations and backward stochastic differential equations.
Probab. Theory Relat. Fields  (2010) 146, 291.
\bibitem{ng} 
N.T. Dung, N. Privault, G.L. Torrisi, Gaussian estimates for the solutions of some one-dimensional stochastic equations. Potential Anal (2015) 43, pp 289--311.
\bibitem{fanwu}
X. Fan, J-L. Wu. Density estimates for the solutions of backward
stochastic differential equations driven by Gaussian processes. 
Potential Anal (2021) 54, pp. 483--501. 
\bibitem{HarterRichou}
J. Harter, A. Richou, A stability approach for solving multidimensional quadratic BSDEs
Electron. J. Probab. (2019) 24, pp. 1--51
\bibitem{ghosh} R. Ghosh,  S.J. Tabrizi, Gene suppression approaches to neurodegeneration. Alz Res Therapy 9,  (2017) pp. 82--94.
\bibitem{Gil}
D. Gillespie, Exact stochastic simulation of coupled chemical reactions,
J. Phys. Chem. 81, (1977) pp. 2340--2361.
\bibitem{jeanblanc} 
M. Jeanblanc, M. Yor, M. Chesney, Mathematical Methods for Financial Markets,
Springer, 2009.
\bibitem{lee}
W.-R. Lee, J.-Y. Jang, J.-S. Kim, M.-H. Kwon, and Y.-S. Kim,
Gene silencing by cell-penetrating, sequence-selective and nucleic-acid hydrolyzing antibodies,
Nucleic Acids Research, Vol. 38, No. 5, (2010) pp. 1596--1609.
\bibitem{Karo}
 N. El Karoui, S. Peng, and M. C. Quenez, Backward stochastic differential equations in
finance, Math. Finance, (1997) :7  1--71.
\bibitem{li} H. Li, Y. Yang, W. Hong, et al. 
Applications of genome editing technology in the targeted therapy of human diseases: mechanisms, advances and prospects. Sig Transduct Target Ther 5, 1 (2020).
\bibitem{TM}
T. Mastrolia, D. Possama\"i, A. R\'eveillac,
{Density Analysis for FBSDEs}, The Annals of Probability
(2016), 44(4), pp. 2817--2857.
\bibitem{TM1}
T. Mastrolia,
{Density analysis of non-Markovian BSDEs and applications to biology and finance}, Stoch Proc Appl,
128, 3, (2018) pp. 897--938.
\bibitem{miller} C. M. Miller, E. N. Harris, Antisense Oligonucleotides: Treatment Strategies and Cellular Internalization, RNA Dis. 2016; 3(4): e1393
\bibitem{NV}
I. Nourdin, F. Viens, Density formula and concentration inequalities with Malliavin calculus. 
Electron. J. Probab. 14, 2287--2309 (2009)
\bibitem{nualart}
{D. Nualart}, The Malliavin calculus and related topics, Springer-Verlag Berlin Heidelberg 2006.
\bibitem{privault} N. Privault. Stochastic Analysis in Discrete and Continuous Settings, volume 1982 of Lecture Notes in Mathematics, p. 309. Springer, Berlin (2009)
\bibitem{OS} C. Olivera and E. Shamarova, Gaussian density estimates for solutions of fully coupled forward-backward SDEs,
Mathematische Nachrichten, 294, pp. 1230--1242, (2020)
\bibitem{rinaldi}
C. Rinaldi,   M. J. A. Wood, Antisense oligonucleotides: the next frontier for treatment of neurological disorders. Nature Reviews Neurology, 14(1),   pp. 9--21, (2017).
\bibitem{gene} E. Shamarova, R. Chertovskih, A. F. Ramos, and P. Aguiar,   
Backward-stochastic-differential-equation approach to modeling of gene expression, 
Physical Review E, v. 95, p. 032418 (2017).
\bibitem{Wazewski} T. Wazewski, Systemes des equations et des inegalites differentielled aux deuxieme membres et leurs applications, Annales Polonici Mathematici, vol. 23, pp. 112--166, 1950.
\bibitem{Yong} J. Yong and X. Y. Zhou, {Stochastic Controls. Hamiltonian Systems and HJB Equations},
Springer, New York, 1999.
\bibitem{Yan} Yanenko N.N. {\em The Method of Fractional Steps. The Solution of Problems of Mathematical Physics in Several Variables}, Springer, 1971.
\end{thebibliography}
\end{document}